\documentclass[a4paper]{amsart}
\usepackage{amssymb,stmaryrd,enumerate}

\newtheorem*{thma}{Theorem~A}
\newtheorem*{thmb}{Theorem~B}
\newtheorem*{thmc}{Theorem~C}

\newtheorem{thm}{Theorem}[section]
\newtheorem{fact}[thm]{Fact}
\newtheorem{lemma}[thm]{Lemma}
\newtheorem{cor}[thm]{Corollary}
\newtheorem{prop}[thm]{Proposition}
\newtheorem{claim}{Claim}[thm]
\newtheorem{subclaim}{Subclaim}[claim]
\theoremstyle{definition}
\newtheorem{defn}[thm]{Definition}
\newtheorem{convention}[thm]{Convention}
\newtheorem{subdefn}{Definition}[claim]
\theoremstyle{remark}
\newtheorem{remark}[thm]{Remark}
\newtheorem*{uremark}{Remark}
\newtheorem{question}[thm]{Question}

\DeclareMathOperator{\trcl}{trcl}
\DeclareMathOperator{\clps}{clps}
\DeclareMathOperator{\reflexive}{Reflexive}
\DeclareMathOperator{\transitive}{Transitive}
\DeclareMathOperator{\baire}{fnc}
\DeclareMathOperator{\ord}{OR}
\DeclareMathOperator{\cf}{cf}
\DeclareMathOperator{\otp}{otp}
\DeclareMathOperator{\acc}{acc}

\DeclareMathOperator{\dom}{dom}
\DeclareMathOperator{\rng}{Im}
\DeclareMathOperator{\Seq}{Lev}
\DeclareMathOperator{\cof}{cof}
\DeclareMathOperator{\pr}{pr}
\DeclareMathOperator{\dl}{Dl}

\renewcommand{\mid}{\mathrel{|}\allowbreak}
\renewcommand{\restriction}{\mathbin\upharpoonright}

\newcommand{\s}{\subseteq}
\newcommand{\br}{\blacktriangleright}
\newcommand{\diagonal}{\bigtriangleup}
\newcommand*{\axiomfont}[1]{\textsf{\textup{#1}}}
\newcommand{\gj}{\axiomfont{GJ}}
\newcommand{\zf}{\axiomfont{ZF}}

\newcommand{\gch}{\axiomfont{GCH}}
\newcommand{\lcc}{\axiomfont{LCC}}
\newcommand{\redulo}{\hookrightarrow_1}
\newcommand*{\sq}[1]{\mathrel{\le^{#1}}}
\newcommand*{\sqc}[1]{\mathrel{\subseteq^{#1}}}

\title{Inclusion modulo nonstationary}

\author{Gabriel Fernandes}
\author{Miguel Moreno}
\author{Assaf Rinot}
\address{Department of Mathematics, Bar-Ilan University, Ramat-Gan 5290002, Israel.}
\urladdr{http://u.math.biu.ac.il/\textasciitilde zanettg}
\urladdr{http://miguelmath.com/}
\urladdr{http://www.assafrinot.com}


\begin{document}
\begin{abstract} A classical theorem of Hechler asserts that the structure $\left(\omega^\omega,\le^*\right)$ is universal in the sense that
for any $\sigma$-directed poset $\mathbb P$ with no maximal element, there is a $ccc$ forcing extension in which
$\left(\omega^\omega,\le^*\right)$ contains a cofinal order-isomorphic copy of $\mathbb P$.
In this paper, we prove a consistency result concerning the universality of the higher analogue $\left(\kappa^\kappa,\le^S\right)$.

\medskip

\textbf{Theorem.} Assume $\gch$. For every regular uncountable cardinal $\kappa$,
there is a cofinality-preserving $\gch$-preserving forcing extension
in which for every analytic quasi-order $\mathbb Q$ over $\kappa^\kappa$
and every stationary subset $S$ of $\kappa$,
there is a Lipschitz map reducing $\mathbb Q$ to $(\kappa^\kappa,\le^S)$.
\end{abstract}

\maketitle

\section{Introduction}

Recall that a \emph{quasi-order} is a binary relation which is reflexive and transitive.
A well-studied quasi-order over the Baire space $\mathbb N^{\mathbb N}$ is the binary relation $\le^*$ which is defined
by letting, for any two elements $\eta:\mathbb N\rightarrow\mathbb N$ and $\xi:\mathbb N\rightarrow\mathbb N$,
$$\eta\le^*\xi\text{ iff }\{n\in\mathbb N\mid \eta(n)>\xi(n)\}\text{ is finite}.$$

This quasi-order is behind the definitions of cardinal invariants $\mathfrak b$  and $\mathfrak d$ (see~\cite[\S2]{MR2768685}),
and serves as a key to the analysis of \emph{oscillation of real numbers} which is known to have prolific applications to topology, graph theory, and forcing axioms (see~\cite{MR980949}).
By a classical theorem of Hechler \cite{MR0360266}, the structure $(\mathbb N^{\mathbb N},\le^*)$ is universal
in that sense that for any $\sigma$-directed poset $\mathbb P$ with no maximal element, there is a $ccc$ forcing extension in which
$(\mathbb N^{\mathbb N},\le^*)$ contains a cofinal order-isomorphic copy of $\mathbb P$.

In this paper, we consider (a refinement of) the higher analogue of the relation $\le^*$ to the realm of the generalized Baire space $\kappa^\kappa$ (sometimes refered as the higher Baire space),
where  $\kappa$ is a regular uncountable cardinal. This is done by simply replacing the ideal of finite sets with the ideal of nonstationary sets, as follows.\footnote{A comparison of the generalization considered here with the one obtained
by replacing the ideal of finite sets with the ideal of bounded sets may be found in \cite[\S8]{MR1355135}.}

\begin{defn}
Given a stationary subset $S\s\kappa$, we define a quasi-order $\sq S$ over $\kappa^\kappa$
by letting, for any two elements $\eta:\kappa\rightarrow\kappa$ and $\xi:\kappa\rightarrow\kappa$,
$$\eta\sq S\xi\text{ iff }\{\alpha\in S\mid \eta(\alpha)>\xi(\alpha)\}\text{ is nonstationary}.$$
\end{defn}

Note that since the nonstationary ideal over $S$ is $\sigma$-closed, the quasi-order $\le^S$ is well-founded,
meaning that we can assign a \emph{rank} value $\|\eta\|$ to each element $\eta$ of $\kappa^\kappa$.
The utility of this approach is demonstrated in the celebrated work of Galvin and Hajnal \cite{MR0376359} concerning the behavior of the power function over the singular cardinals,
and, of course, plays an important role in Shelah's \textit{pcf theory} (see~\cite[\S4]{MR2768693}).
It was also demonstrated to be useful in the study of partition relations of singular cardinals of uncountable cofinality \cite{MR2494318}.

\medskip

In this paper, we first address the question of how $\sq S$ compares with $\sq{S'}$ for various subsets $S$ and $S'$. It is proved:

\begin{thma} Suppose that $\kappa$ is a regular uncountable cardinal and $\gch$ holds.
Then there exists a cofinality-preserving $\gch$-preserving forcing extension in which for all stationary subsets $S,S'$ of $\kappa$,
there exists a map $f:\kappa^{\le\kappa}\rightarrow2^{\le\kappa}$ such that, for all $\eta,\xi\in\kappa^{\le\kappa}$,
\begin{itemize}
\item[$\bullet$] $\dom(f(\eta))=\dom(\eta)$;
\item[$\bullet$] if $\eta\s\xi$, then $f(\eta)\s f(\xi)$;
\item[$\bullet$] if $\dom(\eta)=\dom(\xi)=\kappa$, then  $\eta\sq S\xi$ iff $f(\eta)\sq{S'}f(\xi)$.
\end{itemize}
\end{thma}

Note that as $\rng(f\restriction\kappa^\kappa)\s 2^\kappa$, the above assertion is non-trivial even in the case $S=S'=\kappa$,
and forms a contribution to the study of lossless encoding of substructures of  $(\kappa^{\le\kappa},\ldots)$ as substructures of $(2^{\le\kappa},\ldots)$ (see, e.g., \cite[Appendix]{paper20}).

\medskip

To formulate our next result --- an optimal strengthening of Theorem~A ---
let us recall a few basic notions from generalized descriptive set theory.
\emph{The generalized Baire space} is the set $\kappa^\kappa$ endowed with
the \emph{bounded topology}, in which a basic open set takes the form
$[\zeta]:=\{\eta\in \kappa^\kappa \mid \zeta\subset \eta\}$, with $\zeta$, an element of $\kappa^{<\kappa}$.
A subset $F\s\kappa^\kappa$ is \emph{closed} iff its complement is open iff there exists a tree $T\s\kappa^{<\kappa}$ such that
$[T]:=\{ \eta\in\kappa^\kappa\mid\forall\alpha<\kappa(\eta\restriction\alpha\in T)\}$ is equal to $F$.
A subset $A\s\kappa^\kappa$ is \emph{analytic} iff there is a closed subset $F$ of the product space $\kappa^\kappa\times \kappa^\kappa$
such that its projection $\pr(F):=\{\eta\in\kappa^\kappa\mid \exists\xi\in\kappa^\kappa~(\eta,\xi)\in F\}$ is equal to $A$.
\emph{The generalized Cantor space} is the subspace $2^\kappa$ of $\kappa^\kappa$ endowed with the induced topology.
The notions of open, closed and analytic subsets of $2^\kappa$, $2^\kappa\times2^\kappa$ and  $\kappa^\kappa\times\kappa^\kappa$
are then defined in the obvious way.

\begin{defn}
The restriction of the quasi-order $\sq S$ to $2^\kappa$ is denoted by $\sqc S$.
\end{defn}

For all $\eta,\xi\in\kappa^\kappa$, denote $\Delta(\eta,\xi):=\min(\{\alpha<\kappa\mid \eta(\alpha)\neq\xi(\alpha)\}\cup\{\kappa\})$.

\begin{defn}
Let $R_1$ and $R_2$ be binary relations over $X_1,X_2\in \{2^\kappa, \kappa^\kappa\}$, respectively.
A function $f:X_1\rightarrow X_2$ is said to be:
\begin{enumerate}[(a)]
\item a \emph{reduction of $R_1$ to $R_2$} iff, for all $\eta,\xi\in X_1$, $$\eta\mathrel{R_1}\xi\text{ iff }f(\eta)\mathrel{R_2}f(\xi).$$
\item \emph{$1$-Lipschitz} iff for all $\eta,\xi\in X_1$,  $$\Delta(\eta,\xi)\le\Delta(f(\eta),f(\xi)).$$
\end{enumerate}

The existence of a function $f$ satisfying (a) and (b) is  denoted by ${R_1}\redulo{R_2}$.
\end{defn}

In the above language, Theorem~A provides a model in which, for all stationary subsets $S,S'$ of $\kappa$,
${\sq S}\redulo{\sqc S'}$.
As $\sq{S}$ is an analytic quasi-order over $\kappa^\kappa$,
it is natural to ask whether a stronger universality result is possible,
namely, whether it is forceable that \emph{any} analytic quasi-order over $\kappa^\kappa$ admits a $1$-Lipschitz reduction to $\sqc{S'}$ for some (or maybe even for all) stationary $S'\s\kappa$.
The answer turns out to be affirmative, hence the choice of the title of this paper.

\begin{thmb} Suppose that $\kappa$ is a regular uncountable cardinal and $\gch$ holds.
Then there exists a cofinality-preserving $\gch$-preserving forcing extension
in which, for every analytic quasi-order $Q$ over $\kappa^\kappa$
and every stationary $S\s\kappa$, ${Q}\redulo{\sqc S}$.
\end{thmb}

\begin{uremark}  The universality statement under consideration is optimal, as ${Q}\redulo{\sqc S}$ implies that $Q$ analytic.
\end{uremark}

The proof of the preceding goes through a new diamond-type principle for reflecting second-order formulas,  introduced here and denoted by $\dl^*_S(\Pi^1_2)$.
This principle is a strengthening of Jensen's $\diamondsuit_S$ and a weakening of Devlin's $\diamondsuit^\sharp_S$.
For $\kappa$ a successor cardinal, we have $\dl^*_S(\Pi^1_2)\Rightarrow\diamondsuit^*_S$ but not $\diamondsuit^*_S\Rightarrow\dl^*_S(\Pi^1_2)$ (see Remark~\ref{rmk38} below).
Another crucial difference between the two is that, unlike $\diamondsuit^*_S$, the principle $\dl^*_S(\Pi^1_2)$ is compatible with the set $S$ being ineffable.

In Section~2, we establish the consistency of the new principle, in fact, proving that it follows from an abstract condensation principle that was introduced and studied in \cite{FHl,HolyWuWelch}.
It thus follows that it is possible to force $\dl^*_S(\Pi^1_2)$ to hold over all stationary subsets $S$ of a prescribed regular uncountable cardinal $\kappa$.
It also follows that, in canonical models for Set Theory (including any $L[E]$ model with Jensen's $\lambda$-indexing which is sufficiently iterable and has no subcompact cardinals), $\dl^*_S(\Pi^1_2)$ holds for every stationary subset $S$ of every regular uncountable (including ineffable) cardinal $\kappa$.

Then, in Section~3, the core combinatorial component of our result is proved:
\begin{thmc} Suppose $S$ is a stationary subset of a regular uncountable cardinal $\kappa$.
If $\dl^*_S(\Pi^1_2)$ holds, then, for every analytic quasi-order $Q$ over $\kappa^\kappa$, ${Q}\redulo{\sqc S}$.
\end{thmc}

\section{A Diamond reflecting second-order formulas}
In \cite{MR683163}, Devlin introduced a strong form of the Jensen-Kunen principle $\diamondsuit^+_\kappa$,
which he denoted by $\diamondsuit^{\sharp}_{\kappa}$, and proved:
\begin{fact}[Devlin, {\cite[Theorem~5]{MR683163}}]\label{devlinsthm} In $L$, for every regular uncountable cardinal $\kappa$ that is not ineffable, $\diamondsuit^\sharp_\kappa$ holds.
\end{fact}
\begin{remark}
A subset $S$ of a regular uncountable cardinal $\kappa$ is said to be
\emph{ineffable} iff, for every sequence $\langle Z_\alpha\mid\alpha\in S\rangle$, there exists a subset $Z\s\kappa$, for which $\{\alpha\in S\mid Z\cap\alpha=Z_\alpha\cap\alpha\}$ is stationary.
Note that the collection of non-ineffable subsets of $\kappa$ forms a normal ideal that contains $\{\alpha<\kappa\mid \cf(\alpha)<\alpha\}$ as an element.
Also note that if $\kappa$ is ineffable, then $\kappa$ is strongly inaccessible.
Finally, we mention that by a theorem of Jensen and Kunen, for any ineffable set $S$, $\diamondsuit_S$ holds and $\diamondsuit^*_S$ fails.
\end{remark}

As said before, in this paper, we consider a variation of Devlin's principle compatible with $\kappa$ being ineffable.
Devlin's principle as well as its variation provide us with $\Pi^{1}_{2}$-reflection over structures of the form $\langle \kappa,{\in}, (A_{n})_{n\in \omega} \rangle $.
We now describe the relevant logic in detail.

A $\Pi^{1}_{2}$-sentence $\phi$ is a formula of the form $\forall X\exists Y\varphi$ where $\varphi$ is a first-order sentence over a relational language $\mathcal L$ as follows:
\begin{itemize}
\item[$\bullet$] $\mathcal L$ has a predicate symbol $\epsilon$ of arity $2$;
\item[$\bullet$] $\mathcal L$ has a predicate symbol $\mathbb X$ of arity $m({\mathbb X})$;
\item[$\bullet$] $\mathcal L$ has a predicate symbol $\mathbb Y$ of arity $m({\mathbb Y})$;
\item[$\bullet$] $\mathcal L$ has infinitely many predicate symbols $(\mathbb A_n)_{n\in \omega}$, each $\mathbb A_n$ is of arity $m(\mathbb A_n)$.
\end{itemize}

\begin{defn} For sets $N$ and $x$, we say that \emph{$N$ sees $x$} iff
$N$ is transitive, p.r.-closed, and $x\cup\{x\}\s N$.
\end{defn}

Suppose that a set $N$ sees an ordinal $\alpha$,
and that $\phi=\forall X\exists Y\varphi$ is a $\Pi^{1}_{2}$-sentence, where $\varphi$ is a first-order sentence in the above-mentioned language $\mathcal L$.
For every sequence $(A_n)_{n\in\omega}$ such that, for all $n\in\omega$, $A_n\s \alpha^{m(\mathbb A_n)}$,
we write
$$\langle \alpha,{\in}, (A_{n})_{n\in \omega} \rangle \models_N \phi$$
to express that the two hold:
\begin{enumerate}[(1)]
\item $(A_{n})_{n\in \omega} \in N$;
\item $\langle N,{\in}\rangle\models (\forall X\subseteq \alpha^{m(\mathbb X)})(\exists Y\subseteq \alpha^{m(\mathbb Y)})[\langle \alpha,{\in}, X, Y, (A_{n})_{n\in \omega}  \rangle\models \varphi]$,
where:
\begin{itemize}
\item[$\bullet$] $\in$ is the interpretation of $\epsilon$;
\item[$\bullet$] $X$ is the interpretation of $\mathbb X$;
\item[$\bullet$] $Y$ is the interpretation of $\mathbb Y$, and
\item[$\bullet$] for all $n\in\omega$,  $A_n$ is the interpretation of $\mathbb A_n$.
\end{itemize}
\end{enumerate}
\begin{convention}\label{conv23}
We write $\alpha^+$ for $|\alpha|^+$,
and write $\langle \alpha,{\in}, (A_{n})_{n\in \omega} \rangle \models \phi$ for
$$\langle \alpha,{\in}, (A_{n})_{n\in \omega} \rangle \models_{H_{\alpha^+}} \phi.$$
\end{convention}

\begin{defn}[Devlin, \cite{MR683163}] Let $\kappa$ be a regular and uncountable cardinal.

$\diamondsuit^\sharp_\kappa$ asserts the existence of a sequence $\vec N=\langle N_\alpha\mid\alpha<\kappa\rangle$ satisfying the following:

\begin{enumerate}[(1)]
\item for every infinite $\alpha<\kappa$, $N_\alpha$ is a set of cardinality $|\alpha|$ that sees $\alpha$;
\item for every $X\s\kappa$, there exists a club $C\s\kappa$ such that, for all $\alpha\in C$, $C\cap\alpha,X\cap\alpha\in N_\alpha$;
\item whenever $\langle \kappa,{\in},(A_n)_{n\in\omega}\rangle\models\phi$,
with $\phi$ a $\Pi^1_2$-sentence,
there are stationarily many $\alpha<\kappa$ such that
$\langle \alpha,{\in},(A_n\cap(\alpha^{m(\mathbb A_n)}))_{n\in\omega}\rangle\models_{N_\alpha}\phi$.
\end{enumerate}
\end{defn}

Consider the following variation:

\begin{defn}\label{reflectingdiamond}   Let $\kappa$ be a regular and uncountable cardinal, and $S\s\kappa$ stationary.

$\dl^*_S(\Pi^1_2)$ asserts the existence of a sequence $\vec N=\langle N_\alpha\mid\alpha\in S\rangle$ satisfying the following:

\begin{enumerate}[(1)]
\item for every $\alpha\in S$, $N_\alpha$ is a set of cardinality $<\kappa$ that sees $\alpha$;
\item for every $X\s\kappa$, there exists a club $C\s\kappa$ such that, for all $\alpha\in C \cap S$, $X\cap\alpha\in N_\alpha$;
\item whenever $\langle \kappa,{\in},(A_n)_{n\in\omega}\rangle\models\phi$,
with $\phi$ a $\Pi^1_2$-sentence,
there are stationarily many $\alpha\in S$ such that $|N_\alpha|=|\alpha|$ and
$\langle \alpha,{\in},(A_n\cap(\alpha^{m(\mathbb A_n)}))_{n\in\omega}\rangle\models_{N_\alpha}\phi$.
\end{enumerate}
\end{defn}
\begin{remark} The choice of notation for the above principle is motivated by \cite[Definition~2.10]{Sh:107} and \cite[Definition~45]{TodoVaan}.
\end{remark}

The goal of this section is to derive $\dl^*_S(\Pi^1_2)$ from an abstract principle
which is both forceable and a consequence of $V=L[E]$, for $L[E]$ an iterable extender model with Jensen $\lambda$-indexing without a subcompact cardinal (see~\cite{MR1860606,MR2081183}).
Note that this covers all $L[E]$ models that can be built so far.

\begin{convention}
The class of ordinals is denoted by $\ord$.
The class of ordinals of cofinality $\mu$ is denoted by $\cof(\mu)$, and
the class of ordinals of cofinality greater than $\mu$ is denoted by $\cof({>}\mu)$.
For a set of ordinals $a$, we write
$\acc(a) := \{\alpha \in a  \mid \sup(a \cap \alpha) = \alpha > 0\}$.
$\zf^{-}$ denotes $\zf$ without the power-set axiom.
The transitive closure of a set $X$ is denoted by $\trcl(X)$,
and the Mostowski collapse of a structure $\mathfrak B$ is denoted by $\clps(\mathfrak B)$.
\end{convention}

\begin{defn}\label{nicefiltration}
Suppose $N$ is a transitive set.
For a limit ordinal $\lambda$, we say that $\vec{M}=\langle M_\beta \mid \beta < \lambda \rangle $
is a \emph{nice filtration} of $N$ iff all of the following hold:
\begin{enumerate}[(1)]
\item $\bigcup_{\beta<\lambda}M_\beta=N$;
\item $\vec M$ is $\in$-increasing, that is, $\alpha<\beta<\lambda\implies M_\alpha\in M_\beta$;
\item $\vec M$ is continuous, that is, for every $\beta\in\acc(\lambda)$, $M_\beta=\bigcup_{\alpha<\beta}M_\alpha$;
\item for all $\beta<\lambda$, $M_\beta$ is a transitive set with $M_\beta\cap\ord=\beta$ and $|M_{\beta}|\le|\beta|+\aleph_0$.
\end{enumerate}
\end{defn}

\begin{convention} \label{vecM}
Whenever $\lambda$ is a limit ordinal, and $\vec M=\langle M_\beta\mid\beta<\lambda\rangle$ is a $\s$-increasing, continuous
sequence of sets, we denote its limit $\bigcup_{\beta<\lambda}M_\beta$ by $M_\lambda$.

\end{convention}

\begin{defn}[Holy-Welch-Wu, \cite{HolyWuWelch}]  \label{LCCupto}
Let $\eta < \zeta$ be ordinals.
We say that \emph{local club condensation holds in $(\eta,\zeta)$},
and denote this by $\lcc(\eta,\zeta)$,
iff there exist a limit ordinal $\lambda\ge\zeta$ and a sequence $\vec{M}=\langle M_\beta \mid \beta < \lambda \rangle $ such that
all of the following hold:
\begin{enumerate}[(1)]
\item $\vec M$ is \emph{nice filtration} of $M_{\lambda}$;
\item $\langle M_\lambda,{\in}\rangle \models\zf^{-}$;
\item For every ordinal $\alpha$ in the open interval $(\eta,\zeta)$ and every sequence $\vec{\mathcal{F}} = \langle (F_{n},k_{n}) \mid n \in \omega \rangle$ in $M_\lambda$ such that,
for all $n \in \omega$, $k_{n} \in \omega$ and $F_{n} \subseteq (M_{\alpha})^{k_{n}}$, there is a sequence
$\vec{\mathfrak{B}} = \langle \mathfrak{B}_{\beta} \mid \beta < |\alpha| \rangle $ in $M_\lambda$ having the following properties:
\begin{enumerate}[(a)]
\item for all $\beta<|\alpha|$, $\mathfrak B_{\beta}$ is of the form $$\langle B_{\beta},{\in}, \vec{M} \restriction (B_{\beta} \cap\ord),  (F_n\cap(B_\beta)^{k_n})_{n\in\omega} \rangle;$$
\item for all $\beta<|\alpha|$, $\mathfrak B_{\beta} \prec \langle M_{\alpha},{\in}, \vec{M}\restriction \alpha, (F_n)_{n\in\omega} \rangle$;
\item for all $\beta<|\alpha|$, $\beta\s B_\beta$ and  $|B_{\beta}| < |\alpha|$;
\item for all $\beta < |\alpha|$, there exists $\bar{\beta}<\lambda$ such that
$$\clps(\langle B_{\beta},{\in}, \langle B_{\delta} \mid \delta \in B_{\beta}\cap\ord \rangle \rangle) = \langle M_{\bar{\beta}},{\in}, \vec M\restriction \bar{\beta} \rangle;$$
\item $\langle B_\beta\mid\beta<|\alpha|\rangle$ is $\s$-increasing, continuous and converging to $M_\alpha$.
\end{enumerate}
\end{enumerate}

For $\vec{\mathfrak{B}}$ as in Clause~(3) above we say that
\emph{$\vec{\mathfrak{B}}$ witnesses $\lcc$ at $\alpha$ with respect to $\vec M$ and $\vec{\mathcal{F}}$}.
\end{defn}

\begin{remark}\label{NotationRemark} There are first-order sentences $\psi_0(\dot{\eta},\dot{\zeta})$ and $\psi_{1}(\dot{\eta})$
in the language $\mathcal{L}^*:=\{{\in},\vec{M}, \dot{\eta}, \dot{\zeta}\}$ of set theory augmented by a predicate for a nice filtration and two ordinals such that,
for all $\eta < \zeta \le \lambda$ and $\vec M=\langle M_\beta\mid\beta<\lambda\rangle$:
\begin{itemize}
\item[$\bullet$] $(\langle M_{\lambda},{\in}, \vec{M} \rangle \models \psi_0(\eta,\zeta)) \iff (\vec M\text{ witnesses that }\lcc(\eta,\zeta)\text{ holds})$, and
\item[$\bullet$] $(\langle M_{\lambda},{\in}, \vec{M} \rangle \models \psi_1(\eta)) \iff (\vec M\text{ witnesses that }\lcc(\eta,\lambda)\text{ holds})$.
\end{itemize}
Therefore, we will later make an abuse of notation and write $\langle N,{\in}, \vec{M} \rangle \models\lcc(\eta,\zeta)$
to mean that $\vec M$ is a nice filtration of $N$ witnessing that $\lcc(\eta,\zeta)$ holds.
\end{remark}

\begin{fact}[Friedman-Holy, implicit in  \cite{FHl}] \label{InaccForcing} Assume $\gch$.
For every inaccessible cardinal $\kappa$,
there is a set-size cofinality-preserving notion of forcing $\mathbb{P}$ such that, in $V^{\mathbb{P}}$, the three hold:
\begin{enumerate}[(1)]
\item $\gch$;
\item there is a nice filtration $\vec{M}=\langle M_\beta\mid\beta<\kappa^+\rangle$ of $H_{\kappa^+}$ witnessing that $\lcc(\omega_{1},\kappa^{+})$ holds;
\item there is a $\Delta_{1}$-formula $\Theta$ and a parameter $ a \subseteq \kappa$ such that the relation $<_\Theta$ defined by ($x <_{\Theta} y$ iff $H_{\kappa^{+}} \models \Theta(x,y,a)$) is a global well-ordering of $H_{\kappa^{+}}$.
\end{enumerate}
\end{fact}

\begin{fact}[Holy-Welch-Wu, {\cite[p.~1362 and \S4]{HolyWuWelch}}]\label{forcing}
Assume $\gch$. For every regular cardinal  $\kappa$, there is a set-size notion of forcing $\mathbb{P}$
which is $({<}\kappa)$-directed-closed and has the $\kappa^+{\text{-}}cc$ such that, in $V^{\mathbb P}$, the three hold:
\begin{enumerate}[(1)]
\item $\gch$;
\item there is a nice filtration $\vec{M}=\langle M_\beta\mid\beta<\kappa^+\rangle$ of $H_{\kappa^+}$
witnessing that $\lcc(\kappa,\kappa^{+})$ holds;
\item there is a $\Delta_{1}$-formula $\Theta$ and a parameter $ a \subseteq \kappa$ such that the relation $<_\Theta$ defined by ($x <_{\Theta} y$ iff $H_{\kappa^{+}} \models \Theta(x,y,a)$) is a global well-ordering of $H_{\kappa^{+}}$.
\end{enumerate}
\end{fact}

The following is a improvement of  \cite[Theorem~8]{FHl}.

\begin{fact}[Fernandes, \cite{OnLCC}] \label{NoSubcompact}
Let $L[E]$ be an extender model with Jensen $\lambda$-indexing.
Suppose that, for every $\alpha \in \ord$, the premouse $L[E] || \alpha $ is weakly iterable.\footnote{Here, $L[E]||\alpha$ stands for $\langle J_{\alpha}^{E}, \in, E\restriction \omega \alpha, E_{\omega \alpha} \rangle$,
following the notation from \cite{MR1876087}. For the definition of \emph{weakly iterable}, see \cite[p.~311]{MR1876087}.}
Then, for every infinite cardinal $\kappa$, the following are equivalent:
\begin{enumerate}[(a)]
\item $\langle L_\beta[E]\mid\beta<\kappa^{+} \rangle$ witneses that $\lcc(\kappa^{+},\kappa^{++})$ holds;
\item $L[E] \models``\kappa\text{ is not a subcompact cardinal}"$.
\end{enumerate}
In addition, for every infinite limit cardinal $\kappa$, $\langle L_\beta[E]\mid\beta<\kappa^{+} \rangle$ witnesses that $\lcc(\kappa,\kappa^{+})$ holds.
\end{fact}

\begin{lemma}\label{obvious} Suppose that $\lambda$ is a limit ordinal and that $\vec{M} = \langle M_\beta \mid \beta< \lambda \rangle$ is a nice filtration of $H_\lambda$.
Then, for every infinite cardinal $\theta\le\lambda$, $M_\theta\s H_\theta$.
\end{lemma}
\begin{proof} Let $\theta\le\lambda$ be an infinite cardinal.
By Clause~(4) of Definition~\ref{nicefiltration}, for all $\beta < \theta$, the set $M_{\beta}$ is transitive, $M_{\beta} \cap\ord = \beta$, and $|M_{\beta}|  = |\beta| < \theta $.
It thus follows that $M_{\theta} = \bigcup_{\beta < \theta} M_{\beta} \s H_{\theta}$.
\end{proof}

Motivated by the property of acceptability that holds in extender models, we define the following property for nice filtrations:

\begin{defn} Given a nice filtration  $\vec{M} = \langle M_\beta \mid \beta< \kappa^{+} \rangle$ of $H_{\kappa^{+}}$, we say that $\vec{M}$ is \emph{eventually slow at $\kappa$}
iff there exists an infinite cardinal $\mu < \kappa$ such that, for every cardinal $\theta$ with $\mu<\theta\le\kappa$, $M_{\theta} = H_{\theta}$.
\end{defn}

\begin{lemma}\label{MvsH}
Suppose that $\vec{M}=\langle M_{\beta} \mid \beta < \kappa^{+} \rangle $ is a nice filtration of $H_{\kappa^{+}}$ that is eventually slow at $\kappa$.
Then,  for a tail of $\alpha < \kappa$,
for every sequence $\vec{\mathcal{F}} = \langle (F_{n},k_{n}) \mid n \in \omega \rangle$ such that, for all $n \in \omega$, $k_{n} \in \omega$ and $F_{n} \subseteq (M_{\alpha^+})^{k_{n}}$,
there is $\vec{\mathfrak{B}}$ that witnesses $\lcc$ at $\alpha^+$ with respect to $\vec M$ and $\vec{\mathcal F}$.
\end{lemma}
\begin{proof} Fix an infinite cardinal $\mu < \kappa$ such that, for every cardinal $\theta$ with $\mu<\theta\le\kappa$, $M_{\theta} = H_{\theta}$.
Let $\alpha \in (\mu,\kappa)$ be arbitrary.
Now, given a sequence $\vec{\mathcal{F}}$ as in the statement of the lemma,
build by recursion a $\s$-increasing and continuous sequence $\langle\mathfrak{A}_{\gamma} \mid \gamma< \alpha^{+}\rangle$
of elementary submodels of $\langle M_{\alpha^+},{\in},\vec{M}\restriction\alpha^+, (F_n)_{n\in\omega} \rangle$, such that:
\begin{itemize}
\item[$\bullet$] for each $\gamma < \alpha^{+}$, $|A_{\gamma}|<\alpha^{+}$, and
\item[$\bullet$] $\bigcup_{\gamma < \alpha^{+}} A_{\gamma} = H_{\alpha^{+}}$.
\end{itemize}

By a standard argument, $C:=\{ \gamma<\alpha^+\mid A_\gamma=M_\gamma\}$ is a club in $\alpha^+$.
Let $\{ \gamma_\beta\mid\beta<\alpha^+\}$ denote the increasing enumeration of $C$. Denote $\mathcal B_\beta:=\mathcal A_{\gamma_\beta}$.
Then $\vec{\mathfrak B}=\langle\mathfrak{B}_\beta \mid \beta<\alpha^+\rangle$
is an $\in$-increasing and continuous sequence of elementary submodels of $\langle M_{\alpha^{+}},{\in},\vec M\restriction\alpha^+,(F_n)_{n\in\omega}\rangle$, such that,
for all $\beta < \alpha^{+}$, $\clps(\mathfrak B_{\beta}) = \langle M_{\gamma_{\beta}},{\in},\ldots\rangle$.
\end{proof}

In the next two lemmas we find sufficient conditions for nice filtrations $\langle M_\beta \mid \beta  < \kappa^{+} \rangle $ to be eventually slow at $\kappa$.

\begin{lemma}\label{MH2} Suppose that $\kappa$ is a successor cardinal and that $\vec{M}=\langle M_\beta\mid\beta<\kappa^+\rangle$ is
a nice filtration of $H_{\kappa^+}$ witnessing that  $\lcc(\kappa,\kappa^{+})$ holds.
Then  $\vec{M}$ is eventually slow at $\kappa$.
\end{lemma}
\begin{proof}
As $\kappa$ is a successor cardinal, $\vec{M}$ is eventually slow at $\kappa$ iff $M_{\kappa}=H_{\kappa}$.
Thus, by Lemma~\ref{obvious}, it suffices to verify that $H_{\kappa}\s M_{\kappa}$.
To this end, let $x \in H_{\kappa}$, and we will find $\beta < \kappa$ such that $ x \in M_{\beta}$.

Set $\theta := |\trcl\{x\}|$ and fix a witnessing bijection $f:\theta \leftrightarrow \trcl\{x\}$.
As $H_{\kappa^{+}}=M_{\kappa^{+}}=\bigcup_{\alpha< \kappa^{+}} M_{\alpha}$, we may fix $\alpha < \kappa^+$ such that $\{f,\theta,\trcl\{x\}\} \subseteq M_{\alpha}$.
Let $\vec{\mathfrak{B}}$ witness $\lcc(\kappa,\kappa^+)$ at $\alpha$ with respect to $\vec M$ and $\vec{\mathcal{F}}:=\langle (f,2) \rangle$.
Let $\beta<\kappa^+$ be such that $\clps(\mathfrak B_{\theta+1})=\langle M_{\beta},{\in},\ldots\rangle$.
\begin{claim} $\theta<\beta<\kappa$.
\end{claim}
\begin{proof}
By Definition~\ref{LCCupto}(3)(c), $\theta+1\s B_{\theta+1}$, so that, $\theta<\beta$.
By Clause~(4) of Definition~\ref{nicefiltration} and by Definition~\ref{LCCupto}(3)(c), $|\beta|=|M_\beta|=|B_{\theta+1}|<|\alpha|\le\kappa$.
\end{proof}

Now, as $$\mathfrak B_{\theta+1}\prec  \langle H_{\kappa^{+}},{\in}, \vec{M},F_0 \rangle \models \exists y(\forall\alpha\forall\delta(F_{0}(\alpha,\delta) \leftrightarrow (\alpha,\delta) \in y)),$$
we have $f\in B_{\theta+1}$. Since $\dom(f)\s\ B_{\theta+1}$, $\rng(f)\s B_{\theta+1}$. But $\rng(f)=\trcl(\{x\})$ is a transitive set, so that the Mostowski collapsing map $\pi:B_{\theta+1}\rightarrow M_\beta$ is the identity over $\trcl(\{x\})$,
meaning that  $x\in\trcl(\{x\}) \subseteq M_{\beta}$.
\end{proof}

\begin{lemma}\label{slow} Suppose that $\kappa$ is an inaccessible cardinal, $\mu < \kappa$ and $\vec{M}=\langle M_\beta\mid\beta<\kappa^+\rangle$
witnesses that $\lcc(\mu,\kappa^{+})$ holds.
Then $\mu$ witnesses that $\vec{M}$ is eventually slow at $\kappa$.
\end{lemma}
\begin{proof} Suppose not. It follows from Lemma~\ref{obvious} that we may fix an infinite cardinal $\theta$ with $\mu\le\theta<\kappa$
along with $x \in H_{\theta^{+}} \setminus M_{\theta^{+}}$.
Fix a surjection $f:\theta\rightarrow\trcl(\{x\})$.
Let $\alpha < \kappa^{+}$ be the least ordinal such that $x \in M_{\alpha}$, so that $\mu<\theta^+<\alpha<\kappa^+$.
Let $\vec{\mathfrak{B}}$ witness $\lcc(\mu,\kappa^+)$ at $\alpha$ with respect to $\vec M$ and $\vec{\mathcal{F}}:=\langle (f,2) \rangle$.
Let $\beta<\kappa^+$ be such that $\clps(\mathfrak B_{\theta+1})=\langle M_{\beta},{\in},\ldots\rangle$.
\begin{claim} $\beta<\alpha$.
\end{claim}
\begin{proof}
By Clause~(4) of Definition~\ref{nicefiltration} and by Definition~\ref{LCCupto}(3)(c), $|\beta|=|M_\beta|=|B_{\theta+1}|<|\alpha|$.
and hence $\beta<\alpha$.
\end{proof}

By the same argument used in the proof of Lemma~\ref{MH2}, $x \in M_{\beta}$,
contradicting the minimality of $\alpha$.
\end{proof}

\begin{question}
Notice that if $\kappa$ is an inaccessible cardinal and $\vec{M}=\langle M_{\beta} \mid \beta < \kappa^{+} \rangle $ is such that  $\langle H_{\kappa^{+}}, \in, \vec{M} \rangle \models \lcc(\kappa,\kappa^{+})$,
then, for club many $\beta < \kappa$, $M_{\beta}=H_{\beta}$.
We ask: is it consistent that $\kappa$ is an inaccessible cardinal, $\vec{M}=\langle M_{\beta} \mid \beta < \kappa^{+} \rangle $ is such that $\langle H_{\kappa^{+}},{\in}, \vec{M} \rangle \models \lcc(\kappa,\kappa^{+})$, yet, for stationarily many $\beta < \kappa$, $M_{\beta^{+}} \subsetneq H_{\beta^{+}}$?
\end{question}

\begin{lemma} \label{HvecM1} Suppose that $\vec{M}=\langle M_\beta\mid\beta<\kappa^+\rangle$ is a nice filtration of $H_{\kappa^+}$.
Given a sequence $\vec{\mathcal{F}} = \langle (F_{n},k_{n}) \mid n \in \omega \rangle$ such that, for all $n \in \omega$, $k_{n} \in \omega$ and $F_{n} \subseteq (H_{\kappa^+})^{k_{n}}$,
there are club many $\delta<\kappa^+$ such that
$\langle M_{\delta},{\in}, \vec{M}\restriction \delta, (F_n\cap(M_\delta)^{k_n})_{n\in\omega} \rangle \prec \langle M_{\kappa^+}, {\in}, \vec{M}, (F_n)_{n\in\omega}\rangle $.
\end{lemma}
\begin{proof}   Build by recursion an $\in$-increasing continuous sequence $\vec{\mathfrak B}=\langle\mathfrak{B}_{\beta} \mid \beta < \kappa^{+}\rangle$
of elementary submodels of $\langle M_{\kappa^+},{\in},\vec{M}, (F_n)_{n\in\omega} \rangle$, such that:
\begin{itemize}
\item[$\bullet$] for each $\beta < \kappa^{+}$, $|B_{\beta}|<\kappa^{+}$, and
\item[$\bullet$] $\bigcup_{\beta < \kappa^{+}} B_{\beta} = H_{\kappa^{+}}$.
\end{itemize}

By a standard back-and-forth argument, utilizing the continuity of $\vec{\mathfrak B}$ and $\vec M$, $\{ \delta<\kappa^+\mid B_\delta=M_\delta\}$ is a club in $\kappa^+$.
\end{proof}

\begin{defn}\label{formallcc}
Suppose $\vec{M}=\langle  M_{\beta} \mid \beta < \lambda \rangle $ is a nice filtration of $M_{\lambda}$ for some limit ordinal $\lambda>0$.
Given $\alpha < \lambda$ and $\vec{\mathcal F}=\langle (F_{n},k_{n} ) \mid n \in \omega \rangle$ in $M_{\lambda}$ such that, for each $n \in \omega$, $k_n\in\omega$ and $F_{n} \subseteq(M_{\alpha})^{k_{n}}$,
for every sequence $\vec{\mathfrak{B}} = \langle \mathfrak{B}_{\beta} \mid \beta < |\alpha| \rangle $ in $M_\lambda$
and every letter $l\in\{a,b,c,d,e\}$, we let $\psi_l(\vec{\mathcal B},\vec{\mathcal{F}},\alpha,\vec{M}\restriction(\alpha+1))$ be some formula expressing that Clause~(3)(l) of Definition~\ref{LCCupto} holds.
\end{defn}

The following forms the main result of this section.

\begin{thm}\label{diamond_from_lcc}
Suppose that $\kappa$ is a regular uncountable cardinal,
and $\vec{M}=\langle M_\beta\mid\beta<\kappa^+\rangle$ is a nice filtration of $H_{\kappa^+}$
that is eventually slow at $\kappa$, and witnesses that $\lcc(\kappa,\kappa^+)$ holds.
Suppose further that there  is a subset $ a\s\kappa$ and a formula $ \Theta \in \Sigma_{\omega}$ which defines a well-order $<_\Theta$ in $H_{\kappa^{+}}$ via $ x <_\Theta y $ iff $ H_{\kappa^{+}} \models \Theta(x,y,a)$. Then, for every stationary $ S \subseteq \kappa$, $\dl^*_S(\Pi^{1}_{2})$ holds.
\end{thm}
\begin{proof} Let $S'\s\kappa$ be stationary.
We shall prove that $\dl^*_{S'}(\Pi^{1}_{2})$ holds
by adjusting  Devlin's proof of Fact~\ref{devlinsthm}.

As a first step, we identify a subset $S$ of $S'$ of interest.
\begin{claim} There exists a stationary non-ineffable subset $S\s S'\setminus\omega$ such that, for every $\alpha\in S'\setminus S$, $|H_{\alpha^+}|<\kappa$.
\end{claim}
\begin{proof} If $S'$ is non-ineffable, then let $S:=S'\setminus\omega$, so that $H_{\alpha^+}=H_\omega$ for all $\alpha\in S'\setminus S$.
From now on, suppose that $S'$ is ineffable.
In particular, $\kappa$ is strongly inaccessible and  $|H_{\alpha^+}|<\kappa$ for every $\alpha<\kappa$.
Let $S:=S'\setminus(\omega\cup T)$, where $$T:=\{\alpha\in\kappa\cap\cof({>}\omega)\mid S'\cap\alpha\text{ is stationary in }\alpha\}.$$

To see that $S$ is stationary, let $E$ be an arbitrary club in $\kappa$.

$\br$ If $S'\cap\cof(\omega)$ is stationary, then since $S'\cap\cof(\omega)\s S$, we infer that $S\cap E\neq\emptyset$.

$\br$ If $S'\cap\cof(\omega)$ is non-stationary, then fix a club $C\s E$ disjoint from $S'\cap\cof(\omega)$,
and let $\alpha:=\min(\acc(C)\cap S')$. Then $\cf(\alpha)>\omega$ and $C\cap\alpha$ is a club in $\alpha$ disjoint from $S'$, so that $\alpha\notin T$.
Altogether, $\alpha\in S\cap E$.

To see that $S$ is non-ineffable, we define a sequence $\langle Z_\alpha\mid\alpha\in S\rangle$, as follows.
For every $\alpha\in S$, fix a closed and cofinal subset $Z_\alpha$ of $\alpha$ with $\otp(Z_\alpha)=\cf(\alpha)$ such that, if $\cf(\alpha)>\omega$,
then the club $Z_\alpha$ is disjoint from $S'\cap\alpha$.
Towards a contradiction, suppose that $Z\s\kappa$ is a set for which $\{\alpha\in S\mid Z\cap\alpha=Z_\alpha\}$ is stationary.
Clearly, $Z$ is closed and cofinal in $\kappa$, so that $Z\cap S'$ is stationary, $\otp(Z\cap S')=\kappa$ and hence $D:=\{\alpha<\kappa \mid \otp(Z\cap S'\cap\alpha)=\alpha>\omega\}$ is a club.
Pick $\alpha\in D\cap S$ such that $Z\cap\alpha=Z_\alpha$. As $$\cf(\alpha)=\otp(Z_\alpha)=\otp(Z\cap\alpha)\ge\otp(Z\cap S'\cap\alpha)=\alpha>\omega,$$
it must be the case that $Z_\alpha$ is a club disjoint from $S'\cap\alpha$, while $Z_\alpha=Z\cap\alpha$ and $Z\cap S'\cap\alpha\neq\emptyset$. This is a contradiction.
\end{proof}

Let $S$ be given by the preceding claim.
We shall focus on constructing a sequence $\langle N_{\alpha} \mid \alpha \in S  \rangle $ witnessing $\dl^*_{S}(\Pi^{1}_{2})$
such that, in addition, $|N_\alpha|=|\alpha|$ for every $\alpha\in S$.
It will then immediately follow that the sequence $\langle N'_\alpha\mid\alpha\in S'\rangle$ defined by letting $N_\alpha':=N_\alpha$ for $\alpha\in S$,
and $N'_\alpha:=H_{\alpha^+}$ for $\alpha\in S'\setminus S$ will witness the validity of $\dl^*_{S'}(\Pi^{1}_{2})$. As $\vec M$ is eventually slow at $\kappa$,
we may also assume that, for every $\alpha\in S$, $M_{\alpha^{+}}=H_{\alpha^{+}}$ and the conclusion of Lemma~\ref{MvsH} holds true.\footnote{For all the small $\alpha\in S'\setminus S$ such that $M_{\alpha^+}\neq H_{\alpha^+}$, simply let $N'_\alpha:=N_{\min(S)}$.}
If $\kappa$ is a successor cardinal, we may moreover assume that, for every $\alpha\in S$, $M_{\alpha^+}=H_\kappa$.

Here we go.
As $S$ is non-ineffable, fix a sequence $\vec{Z} = \langle Z_{\alpha} \mid \alpha \in S \rangle $ with $Z_\alpha\s\alpha$ for all $\alpha\in S$,
such that, for every $Z\s\kappa$, $\{\alpha\in S\mid Z\cap\alpha=Z_\alpha\}$ is nonstationary.
In the course of the rest of the proof, we shall occasionally take witnesses to $\lcc$ at some ordinal $\alpha$ with respect to
$\vec M$ and a finite sequence $\vec{\mathcal F}=\langle (F_n,k_n)\mid n\in 4\rangle$; for this, we introduce the following piece of notation for any positive $m<\omega$, $X\s(\kappa^+)^m$ and $\alpha<\kappa^+$:
$$\vec{\mathcal F}_{X,\alpha}:=\langle (X\cap\alpha^m,m),(a\cap\alpha,1),(S\cap\alpha,1),(\vec Z\restriction\alpha,2)\rangle.$$

Next, for each $\alpha \in S$, we define $S_{\alpha}$ to be the set of all $\beta \in \alpha^{+}$ satisfying the following list of conditions:
\begin{enumerate}[(i)]
\item $\langle M_{\beta},{\in}, \vec{M} \restriction \beta \rangle \models \lcc(\alpha,\beta)$,\footnote{Note that $\beta$ is not needed to define $\lcc(\alpha,\beta)$ in the structure $\langle M_{\beta},{\in}, \vec{M} \restriction \beta \rangle$.
Indeed, by $\lcc(\alpha,\beta)$ we mean $\psi_{1}(\alpha)$ as in Remark~\ref{NotationRemark}.}
\item $\langle M_{\beta},{\in}\rangle\models \zf^{-} ~  \& ~ \alpha \text{ is the largest cardinal}$,\footnote{In particular, $\langle M_{\beta},{\in}\rangle\models\alpha\text{ is uncountable}$.}
\item $\langle M_{\beta},{\in}\rangle\models \alpha\text{ is regular}\ \&\ S \cap \alpha ~ \text{is stationary}$,
\item $\langle M_{\beta},{\in}\rangle\models  \Theta(x,y,a\cap \alpha) ~ \text{defines a global well-order}$,
\item $ \vec{Z} \restriction (\alpha + 1) \notin M_{\beta}$.
\end{enumerate}
Then, we consider the set
$$D := \{ \alpha \in S \mid S_{\alpha} \neq \emptyset\ \& \ S_{\alpha} \text{ has no largest element} \}.$$

Define a function $f:S\rightarrow\kappa$ as follow.
For every $\alpha \in D $, let $f(\alpha) := \sup(S_{\alpha})$; for every $\alpha\in S\setminus D$,
let $f(\alpha)$ be the least $\beta< \kappa$ such that $M_{\beta}$ sees $\alpha$, and $\vec{Z} \restriction ( \alpha+1 )  \in  M_{\beta}$.

\begin{claim} \label{welldefined}  $f$ is well-defined. Furthermore, for all $\alpha\in S $, $\alpha<f(\alpha) < \alpha^+$.
\end{claim}
\begin{proof} Let $\alpha\in S$ be arbitrary. The analysis splits into two cases:

$\br$  Suppose $\alpha\in D$. As $\alpha \in S $, we have $ \bigcup_{\beta < \alpha^{+}} M_{\beta} = M_{\alpha^{+}}=H_{\alpha^{+}}$, and hence we may find some $\beta<\alpha^{+}$ such that $\vec{Z}\restriction(\alpha+1) \in M_{\beta}$.
Then, condition~(v) in the definition of $ S_{\alpha}$ implies that $\alpha<f(\alpha)\le\beta<\alpha^+$.

$\br$  Suppose $\alpha\notin D$.  As $\alpha\in S$, let us fix $\langle \mathfrak{B}_{\beta} \mid \beta < \alpha^+ \rangle $ that witnesses $\lcc$ at $\alpha^+$ with respect to $\vec M$ and $\vec{\mathcal F}_{\emptyset,\alpha^+}$.
Set $\beta:=\alpha+2$ and fix ${\bar\beta}<\kappa^+$ such that $\clps(\mathcal B_\beta)=\langle M_{\bar\beta},\ldots\rangle$.
As $\beta\s B_\beta$ and $|B_\beta|<\alpha^+$, by Clause~(4) of Definition~\ref{nicefiltration}, $\beta\le{\bar\beta}<\alpha^+$.
In addition, $\vec{Z} \restriction ( \alpha+1 ) \in M_{{\bar\beta}}$ and there exists an elementary embedding from $\langle M_{{\bar\beta}},{\in}\rangle$ to $\langle H_{\alpha^{+}},{\in}\rangle$,
so that $M_{\bar\beta}$ sees $\alpha$.
Altogether, $\alpha<f(\alpha)\le{\bar\beta}<\alpha^+$.
\end{proof}

Define $\vec{N}= \langle N_{\alpha} \mid \alpha \in S  \rangle $ by letting $N_{\alpha} := M_{f(\alpha)}$ for all $\alpha\in S$.
It follows from Definition~\ref{nicefiltration}(4) and the preceding claim that $|N_\alpha|=|\alpha|$ for all $\alpha\in S$.

\begin{claim}\label{claim2163} Let $X\s\kappa$. Then there exists a club $C\s\kappa$ such that, for all $\alpha\in C\cap S$, $X\cap\alpha \in N_\alpha$.
\end{claim}
\begin{proof}  By Lemma~\ref{HvecM1}, we now fix $\delta < \kappa^{+}$ such that $\kappa,S,a \in M_{\delta} $ and $\langle M_{\delta},{\in},\vec M\restriction\delta\rangle \prec \langle M_{\kappa^{+}},{\in}, \vec{M}\rangle$.
Note that $|\delta|=\kappa$. Let $\vec{\mathfrak B}=\langle\mathfrak B_\alpha\mid\alpha<\kappa\rangle$ witness $\lcc$ at $\delta$ with respect to $\vec M$ and $\vec{\mathcal{F}}_{X,\kappa}$.

\begin{subclaim} $C:= \{ \alpha < \kappa \mid B_{\alpha} \cap \kappa = \alpha  \}$ is a club in $\kappa$.
\end{subclaim}
\begin{proof} To see that $C$ is closed in $\kappa$, fix an arbitrary $\alpha<\kappa$ with $\sup(C \cap \alpha)= \alpha >0$.
As $\langle B_\beta\mid\beta<\kappa\rangle$ is $\subseteq$-increasing and continuous, we have
$$ \alpha = \bigcup_{\beta \in (C\cap \alpha)} \beta = \bigcup_{\beta \in (C \cap \alpha)} (B_{\beta}\cap \kappa)  = \bigcup_{\beta < \alpha} (B_{\beta} \cap \kappa) = B_{\alpha} \cap \kappa.$$

To see that $C$ is unbounded in $\kappa$, fix an arbitrary $\varepsilon<\kappa$, and we shall find $\alpha\in C$ above $\varepsilon$.
Recall that, by Clause~(3)(c) of Definition~\ref{LCCupto}, for each $\beta<\kappa$, $\beta\s B_\beta$ and $|B_\beta|<\kappa$.
It follows that we may recursively construct an increasing sequence of ordinals $ \langle \alpha_{n} \mid n < \omega \rangle $ such that:
\begin{itemize}
\item[$\bullet$] $\alpha_0:=\sup(B_{\varepsilon} \cap \kappa)$, and, for all $n<\omega$:
\item[$\bullet$]  $\sup(B_{\alpha_n} \cap \kappa)<\alpha_{n+1}<\kappa$.
\end{itemize}
In particular, $\sup(B_{\alpha_{n}} \cap \kappa ) \in \alpha_{n+1}$ for all $n<\omega$.
Consequently, for $\alpha:=\sup_{n<\omega}\alpha_{n}$, we have that $\alpha<\kappa$, and
$$ B_{\alpha} \cap \kappa = \bigcup_{n <\omega} (B_{\alpha_{n}} \cap \kappa) \leq \bigcup_{n<\omega} \alpha_{n+1} \leq \bigcup_{n <\omega} (B_{\alpha_{n+2}} \cap \kappa) = \alpha,$$
so that $\alpha\in C\setminus(\varepsilon+1)$.
\end{proof}
To see that the club $C$ is as sought, let $\alpha \in C\cap S$ be arbitrary,
and we shall verify that $X\cap\alpha\in N_\alpha$.
Let $\beta(\alpha)$ be such that $\clps(\mathfrak B_\alpha)=\langle M_{\beta(\alpha)},{\in},\ldots\rangle$,
and let $j_{\alpha}:M_{\beta(\alpha)} \rightarrow B_{\alpha}$ denote the inverse of the collapsing map.
As $\alpha\in C$,  $j_\alpha(\alpha)=\kappa$,
and  $j^{-1}_{\alpha}(Y) = Y \cap \alpha$ for all $Y \in B_{\alpha} \cap \mathcal{P}(\kappa)$.

\begin{subclaim}\label{nsclaim} For every $\beta<\kappa^+$ such that $\vec Z\restriction(\alpha+1)\in M_\beta$, $\beta>\beta(\alpha)$.
\end{subclaim}
\begin{proof} Suppose not, so that $\vec{Z} \restriction (\alpha+1) \in M_{\beta(\alpha)} $.
As $\langle M_{\delta}, {\in} \rangle \prec \langle M_{\kappa^{+}}, {\in} \rangle $, we infer that
\[ \langle M_{\delta},{\in}\rangle\models \forall Z \subseteq \kappa ~ \exists E \text{ club in } \kappa \ (\forall \gamma \in E \cap S \rightarrow Z\cap \gamma \neq Z_{\gamma}),\]
and hence
\[  \langle M_{\beta(\alpha)},{\in}\rangle\models \forall Z \subseteq \alpha ~  \exists E \text{ club in } \alpha \ (\forall \gamma \in E \cap S \rightarrow Z\cap \gamma \neq Z_{\gamma}).\]
In particular, using $ Z := Z_{\alpha}$, we find some $E$ such that
\[  \langle M_{\beta(\alpha)},{\in}\rangle\models  (E \text{ is a club in } \alpha)\land (\forall \gamma \in E \cap S \rightarrow Z_{\alpha} \cap \gamma \neq Z_{\gamma}).\]

Pushing forward with $ E^* := j_{\alpha}(E)$ and $Z^{*}:= j_{\alpha}(Z_{\alpha})$, we infer that
\[ \langle M_{\delta},{\in}\rangle \models (E^* \text{ is a club in } \kappa) \land (\forall \gamma \in E^* \cap S \rightarrow Z^{*} \cap \gamma \neq Z_{\gamma}).\]

Then $Z^{*}\cap \alpha = j_{\alpha}(Z_{\alpha}) \cap \alpha  = Z_{\alpha}$, and hence $\alpha\notin E^*$ (recall that $\alpha\in S$).
Likewise $E^*\cap\alpha=j_\alpha(E)\cap\alpha=E$,  and hence $\alpha\in\acc(E^*)\s E^*$. This is a contradiction.
\end{proof}

Now, since $\vec{\mathfrak B}$ witnesses $\lcc$ at $\delta$ with respect to $\vec M$ and $\vec{\mathcal{F}}_{X,\kappa}$, for each $Y$ in $\{X,a,S\}$, we have that
$$\langle B_{\alpha},{\in},Y\cap B_\alpha \rangle \prec \langle M_{\kappa^{+}},{\in}, Y \rangle \models \exists y\forall z ((z \in y)\leftrightarrow (z\in\kappa  \land  Y(z))),$$
therefore each of $X,a,S$ is a definable element of $\mathfrak B_{\alpha}$. So, as, for all $Y \in B_{\alpha} \cap \mathcal{P}(\kappa)$, $j^{-1}_{\alpha}(Y) = Y \cap \alpha$,
we infer that $X\cap\alpha$, $a\cap \alpha$, and $S \cap \alpha$ are all in $M_{\beta(\alpha)}$.
We will show that $\beta(\alpha) < f(\alpha)$, from which it will follow that $X \cap \alpha \in N_{\alpha}$.
\begin{subclaim}\label{sc61632} $\beta(\alpha) < f(\alpha)$.
\end{subclaim}
\begin{proof} Naturally, the analysis splits into two cases:

$\br$ Suppose $\alpha \notin D$. By definition of $f(\alpha)$ and by Subclaim~\ref{nsclaim}, $\beta(\alpha) < f(\alpha)$.

$\br$ Suppose $\alpha \in D$.
As $\mathfrak B_\alpha\prec \langle M_{\delta},{\in}, \vec{M}\restriction\delta,X,a,S,\vec Z\rangle $
and $\rng(j_\alpha)=B_\alpha$, we infer that
$j_{\alpha}:M_{\beta(\alpha)}\rightarrow M_\delta$ forms an elementary embedding
from $\langle M_{\beta(\alpha)},{\in},\ldots\rangle$ to $\langle M_{\delta},{\in}, \vec{M}\restriction\delta,X,a,S,\vec Z\rangle $ with $j_{\alpha}(\alpha)=\kappa$.
As $\kappa,S,a\in M_{\delta}$ and $ \langle M_{\delta}, {\in}, M \restriction \delta \rangle \prec \langle M_{\kappa}, {\in}, \vec{M} \rangle$, we have:
\begin{enumerate}[(I)]
\item $\langle M_{\delta},{\in}, \vec{M} \restriction \delta \rangle \models \lcc(\kappa,\delta) $,
\item $\langle M_{\delta},{\in}\rangle\models \zf^{-} ~  \& ~ \kappa \text{ is the largest cardinal}$,
\item $\langle M_{\delta},{\in}\rangle\models \kappa\text{ is regular}\ \&\  S \cap \kappa  \text{ is stationary}$,
\item $\langle M_{\delta},{\in}\rangle\models  \Theta(x,y,a\cap \kappa) ~ \text{defines a global well-order}$.
\end{enumerate}
It now follows that $\beta(\alpha)$ satisfies clauses (i),(ii),(iii) and (iv) of the definition of $S_\alpha$.
Together with Subclaim~\ref{nsclaim}, then, $\beta(\alpha) \in S_{\alpha}$. So,
by definitions of $f$ and $D$, $\beta(\alpha)<f(\alpha)$.
\end{proof}
This completes the proof of Claim~\ref{claim2163}.
\end{proof}

We are left with addressing Clause~(3) of Definition~\ref{reflectingdiamond}.

\begin{claim}\label{claim(3)}
The sequence $\langle N_\alpha\mid\alpha\in S\rangle$ reflects $\Pi^{1}_{2} $-sentences.
\end{claim}
\begin{proof}
We need to show that whenever $\langle \kappa,{\in},(A_n)_{n\in\omega}\rangle\models\phi$,
with $\phi=\forall X\exists Y\varphi$ a $\Pi^1_2$-sentence,
for every club $E\s\kappa$, there is $\alpha\in E\cap S$, such that
$$\langle \alpha,{\in},(A_n\cap(\alpha^{m(\mathbb A_n)}))_{n\in\omega}\rangle\models_{N_{\alpha}}\phi.$$
But by adding $E$ to the list $(A_n)_{n\in\omega}$ of predicates,
and by slightly extending the first-order formula $\varphi$ to also assert that $E$ is unbounded,
we would get that any ordinal $\alpha$ satisfying the above will also satisfy that $\alpha$ is an accumulation point of the closed set $E$, so that $\alpha\in E$.
It follows that if any $\Pi^1_2$-sentence valid in a structure of the form $\langle \kappa,{\in},(A_n)_{n\in\omega}\rangle$
reflects to some ordinal $\alpha'\in S$,
then any $\Pi^1_2$-sentence valid in a structure of the form $\langle \kappa,{\in},(A_n)_{n\in\omega}\rangle$
reflects stationarily often in  $S$.

Consider a $\Pi^1_2$-formula $\forall X\exists Y\varphi$,
with integers $p,q$ such that $X$ is a $p$-ary second-order variable and $Y$ is a $q$-ary second-order variable.
Suppose  $\vec A=(A_n)_{n\in\omega}$ is a sequence of finitary predicates on $\kappa$, and $\langle\kappa,{\in},\vec A\rangle\models \forall X\exists Y \varphi$.
By the reduction established in the proof of Proposition~\ref{Prop2.4} below, we may assume that $\vec A$ consists of a single predicate $A_0$ of arity, say, $m_0$.
Recalling Convention~\ref{conv23} and since $M_{\kappa^+}=H_{\kappa^+}$, this altogether means that
$$\langle\kappa,{\in},A_0\rangle\models_{M_{\kappa^+}} \forall X\exists Y \varphi.$$

Let $\gamma$ be the least ordinal such that $\vec{Z},A_0,S\in M_\gamma$. Note that $\kappa<\gamma<\kappa^+$.
Let $\Delta$ denote the set of all $\delta \le \kappa^{+}$ such that:
\begin{enumerate}[(a)]
\item $\langle M_{\delta},{\in}, \vec{M}\restriction \delta \rangle \models \lcc(\kappa,\delta)$,\footnote{In particular, $\delta>\kappa$.}
\item $\langle M_{\delta},{\in}\rangle\models \zf^{-} \  \&\  \kappa ~ \text{is the largest cardinal}$,
\item $\langle M_{\delta},{\in}\rangle\models \kappa\text{ is regular}\ \&\ S \text{ is stationary in } \kappa$,
\item $\langle M_{\delta},{\in}\rangle\models  \Theta(x,y,a) ~ \text{defines a global well-order}$,
\item $\langle\kappa,{\in},A_0\rangle\models_{M_{\delta}} \forall X\exists Y \varphi$,
\item $\langle M_{\delta},{\in}\rangle\models \vec{Z} ~ \text{witness that} ~ S  ~ \text{is not ineffable}$, and
\item $\delta>\gamma$.
\end{enumerate}

As $\kappa^{+} \in \Delta$,
it follows from Lemma~\ref{HvecM1}
and elementarity that $\otp(\Delta\cap\kappa^+)=\kappa^+$.
Let $\{\delta_n\mid n<\omega\}$ denote the increasing enumeration of the first $\omega$ many elements of $\Delta$.
\begin{subdefn}
Let $T(\vec{M},\kappa,S,a,A_0,\vec{Z},\gamma)$ denote the theory consisting of the following axioms:
\begin{enumerate}[(A)]
\item $\vec{M}$ witness $\lcc(\kappa,\kappa^{+})$,
\item $\zf^{-} \  \&\  \kappa  ~ \text{is the largest cardinal}$,
\item $\kappa\text{ is regular}\ \&\ S  \text{ is stationary in }  \kappa$,
\item $\Theta(x,y, a) ~ \text{defines a global well-order}$,
\item $\langle \kappa,{\in}, A_0\rangle\models \forall X\exists Y \varphi$,
\item $\vec{Z}\text{ witness that } S \text{ is not ineffable}$,
\item $\gamma$ is the least ordinal such that $ \{\vec{Z},A_0, S \}\in \vec{M}(\gamma)  $.
\end{enumerate}
\end{subdefn}

Let $n<\omega$. Since $M_{\delta_n}$ is transitive, standard facts (cf.~\cite[Chapter~3, \S5]{drake}) yield the existence of a formula $\Psi$ in the language $\{\dot{\vec{M}},\in\}$
which is $\Delta_{1}^{\zf^{-}}$, and for all $\delta\in(\gamma,\delta_n)$,
\begin{equation}\tag{$\star_1$}
\begin{array}{c} \label{Psi} \langle M_{\delta},{\in}, \vec{M}\restriction \delta \rangle \models T(\vec M \restriction\delta, \kappa,S,a,A_0,\vec{Z},\gamma)
\\\iff\\
\Psi(\vec M \restriction\delta, \kappa,S,a,A_0,\vec{Z},\gamma)
\\\iff\\
\langle M_{\delta_n},{\in},\vec{M}\restriction \delta_n\rangle\models \Psi(\vec M \restriction\delta, \kappa,S,a,A_0,\vec{Z},\gamma).
\end{array}
\end{equation}

Since $\{\delta_k\mid k<\omega\}$   enumerates the first $\omega$ many elements of $\Delta$,
$M_{\delta_{n}}$ believes that there are exactly $ n$ ordinals $\delta$ such that Clauses (a)--(g)  hold for $M_{\delta}$. In fact,
\begin{equation}\label{deltas}\tag{$\star_2$}
\langle M_{\delta_{n}},{\in}, \vec{M} \restriction\delta_{n} \rangle \models   \{\delta \mid  \Psi(\vec M \restriction\delta, \kappa,S,a,A_0, \vec{Z},\gamma) \} = \{\delta_{k} \mid k <n \}.
\end{equation}

Next, for every $ n  < \omega $, as $\langle M_{\delta_{n+1}},{\in}\rangle\models |\delta_{n}|=\kappa$, we may fix in $M_{\delta_{n+1}}$ a sequence
$\vec{\mathfrak{B}_{n}}= \langle \mathcal B_{n,\alpha} \mid \alpha < \kappa \rangle $ witnessing $\lcc$ at $\delta_n$ with respect to $\vec{M}\restriction \delta_{n+1}$ and  $\vec{\mathcal F}_{A_0,\kappa}$
such that, moreover,
$$\langle M_{\delta_{n+1}},{\in}, \vec{M}\restriction \delta_{n+1}\rangle \models ``\vec{\mathfrak{B}_{n}} \text{ is the }  {<_\Theta}\text{-least such witness}".\footnote{Recalling Definition~\ref{formallcc}, this means that
$\langle M_{\delta_{n+1}},{\in}, \vec{M}\restriction \delta_{n+1}\rangle \models ``\vec{\mathfrak{B}_{n}} \text{ is the }  {<_\Theta}\text{-least }\vec{\mathcal{B}}\text{ such that }(\psi_{a} \wedge \psi_{b} \wedge \psi_{c} \wedge \psi_{d} \wedge \psi_{e})(\vec{\mathcal{B}},\vec{\mathcal{F}}_{A_0,\kappa},\delta_n,\vec{M}\restriction(\delta_n+1))"$.}$$

For every $n<\omega$, consider the club $ C_{n}: = \{ \alpha < \kappa \mid B_{n,\alpha} \cap \kappa = \alpha\}$,
and then let $$\alpha':=\min ((\bigcap\nolimits_{n\in \omega} C_{n}) \cap S).$$

For every $n<\omega$, let $\beta_n $ be such that $\clps(\mathfrak B_{n,\alpha'})=\langle M_{\beta_n},{\in},\ldots\rangle$,
and let $j_{n}:M_{\beta_{n}} \rightarrow B_{n,\alpha'} $ denote the inverse of the Mostowski collapse.

\begin{subclaim} Let $n \in \omega$. Then $j_{n}^{-1}(\gamma) = j_{0}^{-1}(\gamma)$.
\end{subclaim}
\begin{proof}
Since
$j^{-1}_{n}(\vec Z)  = \vec Z \restriction \alpha'$, $j^{-1}_{n}(A_0)  = A_0 \cap (\alpha')^{m_0}$  and  $ j^{-1}_{n}(S)  = S \cap \alpha'$,
it follows from
$$\langle M_{\delta_{n}}, {\in} , \vec{M} \restriction \delta_{n} \rangle \models \gamma  \text{ is the least ordinal with } \{ \vec{Z}, A_0, S  \} \subseteq M_{\gamma},$$
that
$$\langle M_{\beta_{n}},{\in}, \vec{M} \restriction \beta_{n} \rangle \models j^{-1}_{n}(\gamma) \text{ is the least ordinal with } \{ \vec{Z}\restriction \alpha', A_0\cap (\alpha')^{m_0}, S \cap \alpha' \} \subseteq M_{\gamma}.$$

Now, let  $\bar{\gamma}$ be such that
$$\langle M_{\beta_{0}},{\in}, \vec{M} \restriction \beta_0 \rangle \models  \bar{\gamma} \text{ is the least ordinal such that } \{ \vec{Z}\restriction \alpha', A_0\cap (\alpha')^{m_0}, S\cap \alpha'  \} \subseteq M_{\bar{\gamma}}. $$
Since $\vec{M}$ is continuous, it follows that $\bar{\gamma}$ is a successor ordinal, that is, $\bar{\gamma}=\sup(\bar\gamma)+1$.
So $\langle M_{\beta_{0}},{\in}, \vec{M} \restriction \beta_0 \rangle $ satisfies the conjunction of the two:
\begin{itemize}
\item[$\bullet$] $\{ \vec{Z}\restriction \alpha', A_0\cap (\alpha')^{m_0}, S\cap \alpha'  \} \subseteq  M_{\bar{\gamma}}$, and
\item[$\bullet$] $\{ \vec{Z}\restriction \alpha', A_0\cap (\alpha')^{m_0}, S\cap \alpha'  \} \not\subseteq  M_{\sup(\bar\gamma)}$.
\end{itemize}
But the two are $\Delta_{0}$-formulas in the parameters $\vec{Z}\restriction \alpha', A_0\cap (\alpha')^{m_0}, S\cap \alpha'  ,M_{\bar\gamma}$ and $M_{\sup(\bar{\gamma})}$, which are all elements of $M_{\beta_0}$.
Therefore,
$$ \langle M_{\beta_{n}},{\in}, \vec{M} \restriction \beta_n \rangle \models \bar{\gamma} \text{ is the least ordinal such that } \{ \vec{Z}\restriction \alpha', A_0\cap (\alpha')^{m_0}, S\cap \alpha'  \} \subseteq M_{\gamma} ,$$
so that $j^{-1}_{n}(\gamma) = \bar{\gamma}= j^{-1}_0(\gamma)$.
\end{proof}

Denote $ \bar{\gamma}:=j^{-1}_0(\gamma)$.
Let $\Psi$ be the same formula used in statement~\eqref{Psi}.
For all $n < \omega$ and $\bar\beta\in(\bar\gamma,\beta_n)$,
setting $\beta:=j_n(\bar{\beta})$, by elementarity of $j_n$:

\begin{equation} \label{PsiBeta}\tag{$\star_3$}
\begin{array}{c} \langle M_{\beta_n},{\in},\vec{M}\restriction \beta_n\rangle\models \Psi(\vec{M} \restriction
\bar{\beta}, \alpha', S\cap\alpha', a\cap \alpha', A_0\cap (\alpha')^{m_0},\vec{Z}\restriction \alpha',\bar{\gamma})
\\ \iff \\
\langle M_{\delta_n}, {\in},\vec{M}\restriction \delta_n \rangle \models \Psi(\vec{M}\restriction \beta,\kappa,S, a, A_0,\vec{Z}, \gamma).
\end{array}
\end{equation}

Hence, for all $n<\omega$, by statements \eqref{deltas} and \eqref{PsiBeta}, it follows that
$$\begin{aligned} \langle M_{\beta_n},{\in}, \vec{M} \restriction\beta_n \rangle
\models~ &
\{\beta   \mid  \Psi(\vec{M} \restriction \beta, \alpha', S \cap \alpha',a\cap \alpha', A_0\cap (\alpha')^{m_0},\vec{Z}\restriction \alpha',\bar{\gamma}) \}  \\
&= \{ j_n^{-1} (\delta_{k}) \mid k < n \},
\end{aligned}$$
and that, for each $k< n $, $j_{n}(\beta_{k})=\delta_{k}$.

\begin{subclaim}$\beta':=\sup_{n \in \omega} \beta_n$ is equal to $\sup(S_{\alpha'})$.
\end{subclaim}
\begin{proof}
For each $n<\omega$, as $\clps(\mathfrak B_{n,\alpha'})=\langle M_{\beta_n},{\in},\ldots\rangle$,
the proof of Subclaim~\ref{sc61632}, establishing that $\beta(\alpha)\in S_\alpha$,
makes clear that $\beta_n \in S_{\alpha'}$.

We first argue that $\beta'\notin S_{\alpha'}$ by showing that $\langle M_{\beta'},{\in}\rangle\not\models\zf^-$,
and then we will argue that no $\beta>\beta'$ is in $S_{\alpha'}$.
Note that  $\{\beta_n \mid n < \omega \}$ is a definable subset of $\beta'$ since it can be defined as the first $\omega$ ordinals to satisfy Clauses (a)--(g),
replacing $\vec M \restriction\delta, \kappa,S,a,A_0,\vec{Z},\gamma$ by $\vec{M} \restriction \beta, \alpha', S \cap \alpha',a\cap \alpha', A_0\cap (\alpha')^{m_0},\vec{Z}\restriction \alpha',\bar{\gamma}$, respectively.
So if $\langle M_{\beta'},{\in}\rangle$ were to model $\zf^-$, we would have get that $\sup_{n<\omega}\beta_n$ is in $M_{\beta'}$, contradicting the fact that $M_{\beta'}\cap\ord=\beta'$.

Now, towards a contradiction, suppose that there exists $\beta>\beta'$ in $S_{\alpha'}$, and let $\beta$ be the least such ordinal.
In particular, $\langle M_{\beta},{\in}\rangle \models \zf^{-}$, and $\langle \beta_n \mid n < \omega \rangle\in M_{\beta}$,
so that $\langle M_{\beta_{n}} \mid n \in \omega \rangle  \in M_{\beta}$.
We will reach a contradiction to Clause~(iii) of the definition of $S_{\alpha'}$,
asserting, in particular, that $S \cap \alpha'$ is stationary in $\langle M_{\beta},{\in}\rangle$.

For each $n<\omega$, we have that
$\langle M_{\delta_{n+1}},{\in}, \vec{M}\restriction \delta_{n+1}\rangle \models \Phi(C_{n},\delta_n,\vec{\mathfrak{B}_{n}}, \kappa ),$
where $\Phi(C_{n},\delta_{n},\vec{\mathfrak{B_{n}}},\kappa)$ is the conjunction of the following two formulas:
\begin{itemize}
\item[$\bullet$] $C_{n} = \{ \alpha < \kappa \mid B_{n, \alpha} \cap \kappa = \alpha \}$, and
\item[$\bullet$] $\vec{\mathfrak{B}_{n}}\text{ is the }  {<_\Theta}\text{-least witness to } \lcc  \text{ at } \delta_n  \text{ with respect to } \vec{M}\restriction\delta_{n+1} \text{ and } \mathcal F_{A_0,\kappa}$.
\end{itemize}
Therefore, for $ \overline{C_{n}}:=j_{n+1}^{-1}(C_{n})$ and $\overline{\mathfrak{B}_{n}}:=j_{n+1}^{-1}(\vec{\mathfrak{B}_{n}})$,
we have
$$\langle M_{\beta_{n+1}},{\in}, \vec{M}\restriction \beta_{n+1}\rangle \models \Phi(\overline{C_{n}},\beta_n,\overline{\mathfrak{B}_{n}}, \alpha').$$
In particular, $\overline{C_n}=j_{n+1}^{-1}(C_{n}) = C_{n}\cap \alpha' $.
Recalling that  $\alpha'=\min ((\bigcap_{n\in \omega} C_{n}) \cap S)$, we infer that
$\bigcap_{n<\omega} \overline{C_{n}}$ is disjoint from $S \cap \alpha'$.
Thus, to establish that $S\cap\alpha'$ is nonstationary, it suffices to verify the two:
\begin{enumerate}[(1)]
\item $\langle \overline{C_n}\mid n<\omega\rangle$ belongs to $M_\beta$, and
\item for every $n<\omega$, $\langle M_{\beta},{\in}\rangle \models \overline{C_{n}}  ~ \text{is a club in} ~ \alpha'$.
\end{enumerate}

As  $\langle M_{\beta_{n}} \mid n \in \omega \rangle  \in M_{\beta}$, we can define $\langle \overline{\mathfrak{B}}_{n} \mid n \in \omega \rangle$ using that, for all $n \in \omega$,
$$
\begin{aligned} \langle M_{\beta_{n+1}},{\in}, \vec{M}\restriction \beta_{n+1}\rangle \models ``\overline{\mathfrak{B}}_{n}  & \text{ is the }  {<_\Theta}\text{-least witness to}
\\ & \lcc \text{ at } \alpha' \text{ w.r.t. } \vec{M}\restriction \beta_{n+1} \text{ and } \mathcal{F}_{A_{0},\alpha'}".
\end{aligned}
$$

This takes care of Clause~(1), and shows that $\langle M_{\beta_{n+1}},{\in}\rangle\models \overline{C_n}\text{ is a club in }\alpha'$.
Since $M_{\beta}$ is transitive and the formula expressing that $\overline{C_{n}}$ is a club is $\Delta_{0}$,
we have also taken care of Clause~(2).
\end{proof}

It follows that $\alpha'\in D$ and $f(\alpha')=\sup(S_{\alpha'})=\beta'$.\footnote{Notice that the argument of this claim also showed that  $D$ is stationary.}
Finally, as, for every $n<\omega$, we have
$$\langle \alpha',{\in}, A_0 \cap(\alpha')^{m_0}\rangle \models_{M_{\beta_n}} \forall X \exists Y \varphi,$$
we infer that $N_{\alpha'}= M_{f(\alpha')}=M_{\beta'} = \bigcup_{n\in\omega} M_{\beta_{n}}$ is such that
$$\langle \alpha',{\in}, A_0 \cap(\alpha')^{m_0} \rangle \models_{N_{\alpha'}} \forall X \exists Y \varphi.$$
Indeed, otherwise there is $X_{0} \in [\alpha']^p \cap N_{\alpha'}$ such that, for all $Y \in [\alpha']^q\cap N_{\alpha'}$,
${N_{\alpha'}}\models[\langle \alpha',{\in}, A_{0}\cap(\alpha')^{m_0} \rangle \models \neg \varphi(X_0,Y)]$.
Find a large enough $n <\omega$ such that $ X_{0} \in M_{\beta_n}$.
Now, since $``\langle \alpha',{\in}, A_{0} \cap (\alpha')^{m_{0}} \rangle \models \neg \varphi(X_0,Y)"$ is a $\Delta_{1}^{\zf^{-}}$ formula on the parameters $ \langle \alpha',{\in}, A_{0} \cap (\alpha')^{m_{0}} \rangle$, $\varphi$,
and since $M_{\beta_n}$ is transitive subset of $N_{\alpha'}$
it follows that, for all $Y \in [\alpha']^q \cap M_{\beta_n}$,
$M_{\beta_n} \models [\langle \alpha',{\in},A_{0}\cap(\alpha')^{m_0} \rangle \models \neg \varphi(X_{0},Y)]$, which is a contradiction.
\end{proof}
This completes the proof of Theorem~\ref{diamond_from_lcc}.
\end{proof}

As a corollary we have found a strong combinatorial axiom that holds everywhere (including at ineffable sets) in canonical models of Set Theory (including G\"odel's constructible universe).
\begin{cor} Suppose that:
\begin{itemize}
\item[$\bullet$] $L[E]$ is an extender model with Jensen $\lambda$-indexing;
\item[$\bullet$] $L[E]\models``\text{there are no subcompact cardinals}"$;
\item[$\bullet$] for every $\alpha \in \ord $, the premouse $L[E]||\alpha$ is weakly iterable.
\end{itemize}
Then, in $L[E]$, for every regular uncountable cardinal $\kappa$, for every stationary $S\s\kappa$, $\dl^*_S(\Pi^1_2)$ holds.
\end{cor}
\begin{proof} Work in $L[E]$.
Let $\kappa$ be any regular and uncountable cardinal.
By Fact~\ref{NoSubcompact}, $\vec M=\langle L_{\beta}[E] \mid \beta < \kappa^{+} \rangle$
witnesses that $\lcc(\kappa,\kappa^{+})$ holds.
Since $L_{\kappa^{+}}[E]$ is an acceptable $J$-structure,\footnote{For the definition of acceptable $J$-structure, see \cite[p.~4]{MR1876087}.}
$\vec{M}$ is a nice filtration of $L_{\kappa^{+}}[E]$ that is eventually slow at $\kappa$.
In addition (cf.~\cite[Lemma~1.11]{MR2768688}),
there is a $\Sigma_{1}$-formula $\Theta$ for which
$$x <_\Theta y  \text{ iff }  L[E]|\kappa^{+}  \models \Theta(x,y)$$  defines a well-ordering of $L_{\kappa^{+}}[E]$.
Finally, acceptability implies that $L_{\kappa^{+}}[E]=H_{\kappa^{+}}$.
Now, appeal to Theorem~\ref{diamond_from_lcc}.
\end{proof}

\section{Universality of inclusion modulo nonstationary}
Throughout this section, $\kappa$ denotes a regular uncountable cardinal satisfying $\kappa^{<\kappa}=\kappa$.
Here, we will be proving Theorems B and C.
Before we can do that,
we shall need to establish a \emph{transversal lemma},
as well as fix some notation and coding that will be useful when working with structures of the form $\langle \kappa,{\in}, (A_n)_{n\in \omega} \rangle $.

\begin{prop}[Transversal lemma]\label{Prop2.4} Suppose that $\langle N_\alpha\mid\alpha\in S\rangle$ is a $\dl^*_S(\Pi^1_2)$-sequence,
for a given stationary $S\s\kappa$.
For every $\Pi^1_2$-sentence $\phi$,
there exists a transversal $\langle \eta_\alpha\mid\alpha\in S\rangle\in\prod_{\alpha\in S}N_\alpha$ satisfying the following.

For every $\eta\in\kappa^\kappa$,
whenever $\langle \kappa,{\in},(A_n)_{n\in\omega}\rangle\models \phi$,
there are stationarily many $\alpha\in S$ such that
\begin{enumerate}[(i)]
\item $\eta_\alpha=\eta\restriction\alpha$, and
\item $\langle \alpha,{\in},(A_n\cap (\alpha^{m(\mathbb A_n)}))_{n\in\omega}\rangle\models_{N_\alpha}\phi$.
\end{enumerate}
\end{prop}
\begin{proof}
Let $c:\kappa \times \kappa \leftrightarrow \kappa$ be some primitive-recursive pairing function.
For each $\alpha\in S$, fix a surjection $f_\alpha:\kappa\rightarrow N_\alpha$ such that $f_\alpha[\alpha]=N_\alpha$ whenever $|N_\alpha|=|\alpha|$.
Then, for all $i<\kappa$, as $f_\alpha(i)\in N_\alpha$, we may define a set $\eta^i_\alpha$ in $N_\alpha$ by letting
$$\eta_\alpha^i:=\begin{cases}
\{(\beta,\gamma)\in\alpha\times\alpha\mid c(i,c(\beta,\gamma))\in f_\alpha(i)\},&\text{if }i<\alpha;\\
\emptyset,&\text{otherwise}.
\end{cases}$$

We claim that for every $\Pi^1_2$-sentence $\phi$,
there exists $i(\phi)<\kappa$ for which $\langle \eta_\alpha^{i(\phi)}\mid \alpha\in S\rangle$ satisfies the conclusion of our proposition.
Before we prove this, let us make a few reductions.

First of all, it is clear that for every $\Pi^1_2$-sentence $\phi=\forall X\exists Y\varphi$,
there exists  a large enough $n'<\omega$ such that all predicates mentioned in $\varphi$  are in $\{ \epsilon,\mathbb X,\mathbb Y,\mathbb A_n\mid n<n'\}$.
So the only structures of interest for $\phi$ are in fact $\langle \alpha,{\in},(A_n)_{n<n'}\rangle$, where $\alpha\le\kappa$.
Let $m':=\max\{m(\mathbb A_n)\mid n<n'\}$.
Then, by a trivial manipulation of $\varphi$,
we may assume that the only structures of interest for $\phi$ are in fact $\langle \alpha,{\in},A_0\rangle$, where $\omega\le\alpha\le\kappa$ and $m(\mathbb A_0)=m'+1$.

Having the above reductions in hand, we now fix a $\Pi^1_2$-sentence $\phi=\forall X\exists Y\varphi$ and positive integers $m$ and $k$ such that
the only predicates mentioned in $\varphi$ are in $\{\epsilon,\mathbb X,\mathbb Y,\mathbb A_0\}$,  $m(\mathbb A_0)=m$ and $m(\mathbb Y)= k $.

\begin{claim} There exists $i<\kappa$ satisfying the following.
For all $\eta\in\kappa^\kappa$ and $A\s\kappa^m$,
whenever $\langle \kappa,{\in},A\rangle\models\phi$,
there are stationarily many $\alpha\in S$ such that
\begin{enumerate}[(i)]
\item $\eta^i_\alpha=\eta\restriction\alpha$, and
\item $\langle \alpha,{\in},A\cap(\alpha^m)\rangle\models_{N_\alpha}\phi$.
\end{enumerate}
\end{claim}
\begin{proof}
Suppose not. Then, for every $i<\kappa$, we may fix  $\eta_i\in\kappa^\kappa$, $A_i\s\kappa^m$ and a club $C_i\s\kappa$ such that  $\langle \kappa,{\in},A_i\rangle\models\phi$, but, for all $\alpha\in C_i\cap S$,
one of the two fails:
\begin{enumerate}[(i)]
\item $\eta_\alpha^i=\eta_i\restriction\alpha$, or
\item $\langle \alpha,{\in},A_i\cap(\alpha^m)\rangle\models_{N_\alpha}\phi$.
\end{enumerate}

Let
\begin{itemize}
\item[$\bullet$] $Z:=\{ c(i,c(\beta,\gamma))\mid i<\kappa, (\beta,\gamma)\in \eta_i\}$,
\item[$\bullet$] $A:=\{(i,\delta_1,\ldots,\delta_m)\mid i<\kappa,(\delta_1,\ldots,\delta_m)\in A_i\}$, and
\item[$\bullet$] $C:=\diagonal_{i<\kappa}\{\alpha\in C_i\mid \eta_i[\alpha]\s\alpha\}$.
\end{itemize}
Fix a variable $i$ that does not occur in $\varphi$.
Define a first-order sentence $\psi$ mentioning only the predicates in $\{\epsilon,\mathbb X,\mathbb Y,\mathbb A_1\}$ with $m(\mathbb A_1)=1+m$ and $m(\mathbb Y)=1+k$
by replacing all occurrences of the form $\mathbb A_0(x_1,\ldots,x_m)$ and $\mathbb Y(y_1, \ldots, y_k)$ in $\varphi$ by $\mathbb A_1(i,x_1,\ldots,x_m)$ and $\mathbb Y(i,y_1,\ldots,y_k)$, respectively.
Then, let $\varphi':=\forall i(\psi)$, and finally let $\phi':=\forall X\exists Y\varphi'$, so that $\phi'$ is a $\Pi^1_2$-sentence.

A moment reflection makes it clear that $\langle\kappa,{\in},A\rangle\models\phi'$.
Thus, let $S'$ denote the set of all $\alpha \in S$ such that all of the following hold:
\begin{enumerate}[(1)]
\item $\alpha\in C$;
\item $c[\alpha\times\alpha]=\alpha$;
\item $Z\cap\alpha\in N_\alpha$;
\item $|N_\alpha|=|\alpha|$;
\item $\langle\alpha,{\in},A\cap(\alpha^{m+1})\rangle\models_{N_\alpha}\phi'$.
\end{enumerate}

By hypothesis, $S'$ is stationary.
For all $\alpha\in S'$, by Clauses (3) and (4), we have $Z\cap\alpha\in N_\alpha=f_\alpha[\alpha]$, so, by Fodor's lemma, there exists some $i<\kappa$ and a stationary $S''\s S'\setminus(i+1)$ such that, for all $\alpha\in S''$:
\begin{enumerate}
\item[(3')] $Z\cap\alpha=f_\alpha(i)$.
\end{enumerate}
Let $\alpha\in S''$. By Clause~(5), we in particular have
\begin{enumerate}
\item[(5')] $\langle\alpha,{\in},A_i\cap(\alpha^m)\rangle\models_{N_\alpha}\phi$.
\end{enumerate}
Also, by Clause~(1), we have $\alpha\in C_i$, and
so we must conclude that $\eta_i\restriction\alpha\neq \eta^i_\alpha$.
However, $\eta_i[\alpha]\s\alpha$, and $Z\cap\alpha=f_\alpha(i)$, so that, by Clause~(2),
$$\eta_i\restriction\alpha=\eta_i\cap(\alpha\times\alpha)=\{(\beta,\gamma)\in\alpha\times\alpha\mid c(i,c(\beta,\gamma))\in f_\alpha(i)\}=\eta^i_\alpha.$$ This is a contradiction.
\end{proof}
This completes the proof of Proposition~\ref{Prop2.4}.
\end{proof}

\begin{lemma}\label{prop31}
There is a first-order sentence  $\psi_{\baire} $ in the language with binary predicate symbols $\epsilon$ and $\mathbb X$ such that,
for every ordinal $\alpha$ and every $X\s\alpha\times\alpha$,
$$(X\text{ is a function from }\alpha\text{ to }\alpha)\text{ iff }(\langle \alpha,{\in},X\rangle  \models \psi_{\baire}).$$
\end{lemma}
\begin{proof} Let $\psi_{\baire}:=\forall\beta\exists\gamma(\mathbb X(\beta,\gamma)\land(\forall\delta(\mathbb X(\beta,\delta)\rightarrow\delta=\gamma)))$.
\end{proof}

\begin{lemma}\label{prop32} Let $\alpha$ be an ordinal. Suppose that $\phi$ is a $\Sigma^1_1$-sentence involving a predicate symbol $\mathbb A$ and two binary predicate symbols $\mathbb X_0,\mathbb X_1$.
Denote $R_\phi:=\{ (X_0,X_1)\mid \langle \alpha,{\in},A,X_0,X_1\rangle\models\phi\}$.
Then there are $\Pi^1_2$-sentences $\psi_{\reflexive}$ and $\psi_{\transitive}$ such that:
\begin{enumerate}[(1)]
\item $(R_\phi\supseteq\{(\eta,\eta)\mid \eta\in\alpha^\alpha\})$ iff $(\langle\alpha,{\in},A\rangle\models\psi_{\reflexive})$;
\item  $(R_\phi\text{ is transitive})$ iff $(\langle\alpha,{\in},A\rangle\models\psi_{\transitive})$.
\end{enumerate}
\end{lemma}
\begin{proof}\begin{enumerate}[(1)]
\item Fix a first-order sentence $\psi_{\baire}$ such that ($X_0\in\alpha^\alpha$) iff ($\langle\alpha,{\in},X_0\rangle\models\psi_{\baire}$).
Now, let $\psi_{\reflexive}$ be $\forall X_0\forall X_1((\psi_{\baire}\land(X_1=X_0))\rightarrow\phi)$.

\item Fix a $\Sigma^1_1$-sentence $\phi'$ involving predicate symbols $\mathbb A, \mathbb X_1,\mathbb X_2$
and a $\Sigma^1_1$-sentence $\phi''$ involving binary symbols $\mathbb A,\mathbb X_0,\mathbb X_2$
such that
\begin{gather*}
\{ (X_1,X_2)\mid \langle \alpha,{\in},A,X_1,X_2\rangle\models\phi'\}=\\R_\phi=\{ (X_0,X_2)\mid \langle \alpha,{\in},A,X_0,X_2\rangle\models\phi''\}
\end{gather*}
Now, let $\psi_{\transitive}:=\forall X_0\forall X_1\forall X_2((\phi\land\phi')\rightarrow\phi'')$.\qedhere
\end{enumerate}
\end{proof}

\begin{defn}\label{defn34}
Denote by $ \Seq_3(\kappa)$ the set of level sequences in $\kappa^{<\kappa}$ of length $3$:
$$ \Seq_3(\kappa):=\bigcup_{\tau<\kappa}\kappa^\tau\times\kappa^\tau\times\kappa^\tau.$$

Fix an injective enumeration $\{\ell_\delta\mid\delta<\kappa\}$ of $\Seq_3(\kappa)$.
For each $\delta<\kappa$, we denote $\ell_\delta=(\ell_\delta^0,\ell_\delta^1,\ell_\delta^2)$.
We then encode each $T \subseteq \Seq_{3}(\kappa) $ as a subset of $\kappa^5$ via:
$$T_\ell := \{ (\delta,\beta, \ell_\delta^0(\beta),\ell_\delta^{1}(\beta),\ell_\delta^{2}(\beta)) \mid \delta<\kappa, \ell_\delta \in T, \beta\in \dom(\ell_\delta^0) \}.$$
\end{defn}

We now prove Theorem~C.
\begin{thm}\label{Sigmacompl}
Suppose  $\dl^*_S(\Pi^1_2)$ holds for a given stationary $S\s\kappa$.

For every analytic quasi-order $Q$ over $\kappa^\kappa$,
there is a $1$-Lipschitz map $f:\kappa^\kappa\rightarrow2^\kappa$ reducing $Q$ to $\sqc{S}$.
\end{thm}
\begin{proof}
Let $Q$ be an analytic quasi-order over $\kappa^\kappa$.
Fix a tree $T$ on $\kappa^{<\kappa}\times \kappa^{<\kappa}\times \kappa^{<\kappa} $ such that $Q=\pr([T])$, that is,
$$(\eta,\xi) \in Q \iff \exists \zeta\in\kappa^{\kappa}~ \forall \tau < \kappa ~(\eta\restriction\tau,\xi\restriction\tau,\zeta\restriction\tau) \in T.$$

We shall be working with a first-order language having a $5$-ary predicate symbol $\mathbb A$
and binary predicate symbols $\mathbb X_0,\mathbb X_1,\mathbb X_2$ and $\epsilon$.
By Lemma~\ref{prop31}, for each $i<3$, let us fix a sentence $\psi_{\baire}^i$ concerning the binary predicate symbol $\mathbb X_i$
instead of $\mathbb X$, so that
$$(X_i\in\kappa^{\kappa})\text{ iff }(\langle \kappa,{\in},A,X_0,X_1,X_2\rangle  \models \psi_{\baire}^i).$$
Define a sentence $\varphi_Q$ to be the conjunction of four sentences:
$\psi_{\baire}^0$, $\psi^1_{\baire}$, $\psi^2_{\baire}$, and
$$\forall \tau \exists \delta \forall \beta [\epsilon(\beta,\tau)\rightarrow\exists \gamma_{0}\exists\gamma_{1}\exists\gamma_{2} (\mathbb X_0(\beta,\gamma_{0}) \land \mathbb X_1(\beta,\gamma_{1})\land\mathbb X_2(\beta,\gamma_{2})\land\mathbb A(\delta,\beta,\gamma_{0},\gamma_{1},\gamma_{2}) )].$$
Set $A:=T_\ell$ as in Definition~\ref{defn34}. Evidently, for all $\eta,\xi,\zeta\in\mathcal P(\kappa\times\kappa)$, we get that
$$\langle\kappa,{\in},A,\eta,\xi,\zeta\rangle\models\varphi_Q$$
iff the two hold:
\begin{enumerate}[(1)]
\item $\eta,\xi,\zeta\in\kappa^\kappa$, and
\item for every $\tau<\kappa$, there exists $\delta<\kappa$, such that $\ell_\delta=(\eta\restriction\tau,\xi\restriction\tau,\zeta\restriction\tau)$ is in $T$.
\end{enumerate}
Let $\phi_Q:= \exists X_2(\varphi_Q)$. Then $\phi_Q$ is a $\Sigma^1_1$-sentence involving predicate symbols $\mathbb A,\mathbb X_0,\mathbb X_1$ and $\epsilon$ for which the induced binary relation
$$R_{\phi_Q}:=\{(\eta,\xi)\in (\mathcal P(\kappa\times\kappa))^2\mid \langle \kappa,{\in},A,\eta,\xi\rangle \models \phi_Q\}$$
coincides with the quasi-order $Q$.
Now, appeal to Lemma~\ref{prop32} with $\phi_Q$ to receive the corresponding $\Pi^1_2$-sentences $\psi_{\reflexive}$ and $\psi_{\transitive}$.
Then, consider the following two $\Pi^1_2$-sentences:
\begin{itemize}
\item[$\bullet$] $\psi^0_Q:=\psi_{\reflexive}\land\psi_{\transitive}\land\phi_Q$, and
\item[$\bullet$] $\psi^1_Q:=\psi_{\reflexive}\land\psi_{\transitive}\land\neg(\phi_Q)$.
\end{itemize}

Let $\vec N=\langle N_\alpha\mid\alpha\in S\rangle$ be a  $\dl^*_S(\Pi^1_2)$-sequence.
Appeal to Proposition~\ref{Prop2.4} with the $\Pi^1_2$-sentence $\psi^1_Q$ to obtain a corresponding transversal $\langle \eta_\alpha\mid\alpha\in S\rangle\in\prod_{\alpha\in S}N_\alpha$.
Note that we may assume that, for all $\alpha\in S$, $\eta_\alpha\in{}^\alpha\alpha$, as this does not harm
the key feature of the chosen transversal.\footnote{For any $\alpha$ such that $\eta_\alpha$ is not a function from $\alpha$ to $\alpha$,
simply replace $\eta_\alpha$ by the constant function from $\alpha$ to $\{0\}$.}

For each $\eta\in\kappa^\kappa$, let
$$Z_\eta:=\{\alpha\in S\mid A\cap\alpha^5\text{ and }\eta\restriction\alpha\text{ are in }N_\alpha\}.$$
\begin{claim} Suppose $\eta\in\kappa^\kappa$. Then $S\setminus Z_\eta$ is nonstationary.
\end{claim}
\begin{proof} Fix primitive-recursive bijections $c:\kappa^2\leftrightarrow\kappa$ and $d:\kappa^5\leftrightarrow\kappa$.
Given $\eta\in\kappa^\kappa$, consider the club $D_0$ of all $\alpha<\kappa$ such that:
\begin{itemize}
\item[$\bullet$] $\eta[\alpha]\s\alpha$;
\item[$\bullet$] $c[\alpha\times\alpha]=\alpha$;
\item[$\bullet$] $d[\alpha\times\alpha\times\alpha\times\alpha\times\alpha]=\alpha$.
\end{itemize}

Now, as $c[\eta]$ is a subset of $\kappa$, by the choice $\vec N$, we may find a club $D_1\s\kappa$ such that, for all $\alpha\in D_1\cap S$,
$c[\eta]\cap\alpha\in N_\alpha$.
Likewise, we may find a club $D_2\s\kappa$ such that, for all $\alpha\in D_2\cap S$,
$d[A]\cap\alpha\in N_\alpha$.

For all $\alpha\in S\cap D_0\cap D_1\cap D_2$, we have
\begin{itemize}
\item[$\bullet$] $c[\eta\restriction\alpha]=c[\eta\cap(\alpha\times\alpha)]=c[\eta]\cap c[\alpha\times\alpha]=c[\eta]\cap\alpha\in N_\alpha$, and
\item[$\bullet$] $d[A\cap\alpha^5]=d[A]\cap d[\alpha^5]=d[A]\cap\alpha\in N_\alpha$.
\end{itemize}
As $N_\alpha$ is p.r.-closed, it then follows that $\eta\restriction\alpha$ and $A\cap\alpha^5$ are in $N_\alpha$.
Thus, we have shown that $S \setminus Z_\eta$ is disjoint from the club $D_0\cap D_1\cap D_2$.
\end{proof}
For all $\eta\in\kappa^\kappa$ and $\alpha\in Z_\eta$, let:
$$\mathcal{P}_{\eta,\alpha}:=\{p\in \alpha^\alpha\cap N_\alpha\mid  \langle\alpha,{\in},A\cap\alpha^5,p,\eta\restriction\alpha\rangle \models_{N_{\alpha}} \psi^0_Q \}.$$

Finally, define a function $f:\kappa^{\kappa} \rightarrow 2^\kappa$ by letting, for all $\eta\in\kappa^\kappa$ and $\alpha<\kappa$,
$$f(\eta)(\alpha) := \begin{cases}
1,& \text{if } \alpha\in Z_\eta\text{ and }\eta_\alpha\in \mathcal{P}_{\eta,\alpha};\\
0,             & \text{otherwise}.
\end{cases}
$$

\begin{claim} $f$ is $1$-Lipschitz.
\end{claim}
\begin{proof} Let $\eta,\xi$ be two distinct elements of $\kappa^{\kappa}$.
Let $\alpha\le\Delta(\eta,\xi)$ be arbitrary.

As $\eta\restriction\alpha=\xi\restriction\alpha$, we have $\alpha\in Z_\eta$ iff $\alpha\in Z_\xi$.
In addition, as $\eta\restriction\alpha=\xi\restriction\alpha$, $\mathcal P_{\eta,\alpha}=\mathcal P_{\xi,\alpha}$ whenever  $\alpha\in Z_\eta$.
Thus, altogether, $f(\eta)(\alpha)=1$ iff $f(\xi)(\alpha)=1$.
\end{proof}
\begin{claim}\label{c353} Suppose $(\eta,\xi)\in Q$. Then $f(\eta)\sqc Sf(\xi)$.
\end{claim}
\begin{proof} As $(\eta,\xi)\in Q$, let us fix $\zeta\in\kappa^\kappa$ such that, for all $\tau<\kappa$,
$(\eta\restriction\tau,\xi\restriction\tau,\zeta\restriction\tau)\in T$.
Define a function $g:\kappa\rightarrow\kappa$ by letting, for all $\tau<\kappa$,
$$g(\tau):=\min\{\delta<\kappa\mid \ell_\delta=(\eta\restriction\tau,\xi\restriction\tau,\zeta\restriction\tau)\}.$$
As $(S\setminus Z_\eta)$, $(S\setminus Z_\xi)$ and $(S\setminus Z_\zeta)$ are nonstationary,
let us fix a club $C\s\kappa$ such that $C\cap S\s Z_\eta\cap Z_\xi\cap Z_\zeta$.
Consider the club $D:=\{\alpha\in C\mid g[\alpha]\s\alpha\}$.
We shall show that, for every $\alpha\in D\cap S$, if $f(\eta)(\alpha)=1$ then $f(\xi)(\alpha)=1$.

Fix an arbitrary $\alpha\in D\cap S$ satisfying $f(\eta)(\alpha)=1$. In effect, the following three conditions are satisfied:
\begin{enumerate}[(1)]
\item $\langle\alpha,{\in},A\cap\alpha^5\rangle \models_{N_{\alpha}} \psi_{\reflexive}$,
\item $\langle\alpha,{\in},A\cap\alpha^5\rangle \models_{N_{\alpha}} \psi_{\transitive}$, and
\item $\langle\alpha,{\in},A\cap\alpha^5,\eta_\alpha,\eta\restriction\alpha\rangle \models_{N_{\alpha}} \phi_Q$.
\end{enumerate}

In addition, since $\alpha$ is a closure point of $g$, by definition of $\varphi_Q$, we have
$$\langle\alpha,{\in},A\cap\alpha^5,\eta\restriction\alpha,\xi\restriction\alpha,\zeta\restriction\alpha\rangle\models\varphi_Q.$$

As $\alpha\in S$ and $\varphi_Q$ is first-order,\footnote{$N_{\alpha}$ is transitive and rud-closed (in fact, p.r.-closed), so that  $N_{\alpha} \models \gj$ (see \cite[\S Other remarks on GJ]{Mathias}).
Now, by  \cite[\S The cure in $\gj$, proposition 10.31]{Mathias},  $\mathbf{Sat}$ is $\Delta_{1}^{\gj}$.}
$$\langle\alpha,{\in},A\cap\alpha^5,\eta\restriction\alpha,\xi\restriction\alpha,\zeta\restriction\alpha\rangle\models_{N_\alpha}\varphi_Q,$$
so that, by definition of $\phi_Q$,
$$\langle\alpha,{\in},A\cap\alpha^5,\eta\restriction\alpha,\xi\restriction\alpha\rangle\models_{N_\alpha}\phi_Q.$$
By combining the preceding with clauses (2) and (3) above, we infer that the following holds, as well:
\begin{enumerate}
\item[(4)] $\langle\alpha,{\in},A\cap\alpha^5,\eta_\alpha,\xi\restriction\alpha\rangle \models_{N_{\alpha}} \phi_Q$.
\end{enumerate}
Altogether, $f(\xi)(\alpha)=1$, as sought.
\end{proof}

\begin{claim} Suppose $(\eta,\xi)\in \kappa^\kappa\times\kappa^\kappa\setminus Q$. Then $f(\eta)\not\sqc{S}f(\xi)$.
\end{claim}
\begin{proof} As $(S\setminus Z_\eta)$ and $(S\setminus Z_\xi)$ are nonstationary,
let us fix a club $C\s\kappa$ such that $C\cap S\s Z_\eta\cap Z_\xi$.
As $Q$ is a quasi-order and $(\eta,\xi)\notin Q$, we have:
\begin{enumerate}[(1)]
\item $\langle\kappa,{\in},A\rangle \models\psi_{\reflexive}$,
\item $\langle\kappa,{\in},A\rangle \models\psi_{\transitive}$, and
\item $\langle\kappa,{\in},A,\eta,\xi\rangle \models \neg(\phi_Q)$.
\end{enumerate}
so that, altogether,
$$ \langle \kappa,{\in},A,\eta,\xi\rangle \models \psi^1_Q.$$

Then, by the choice of the transversal $\langle \eta_\alpha\mid\alpha\in S\rangle$, there is a stationary subset $S'\s S\cap C$ such that,
for all $\alpha\in S'$:
\begin{enumerate}[(1')]
\item $\langle\alpha,{\in},A\cap\alpha^5\rangle \models_{N_{\alpha}} \psi_{\reflexive}$,
\item $\langle\alpha,{\in},A\cap\alpha^5\rangle \models_{N_{\alpha}} \psi_{\transitive}$,
\item $\langle\alpha,{\in},A\cap\alpha^5,\eta\restriction\alpha,\xi\restriction\alpha\rangle \models_{N_{\alpha}} \neg(\phi_Q)$, and
\item $\eta_\alpha=\eta\restriction\alpha$.
\end{enumerate}

By Clauses (3') and (4'), we have that $\eta_\alpha\notin\mathcal P_{\xi,\alpha}$, so that $f(\xi)(\alpha)=0$.

By Clauses (1'), (2') and (4'), we have that $\eta_\alpha\in\mathcal P_{\eta,\alpha}$, so that $f(\eta)(\alpha)=1$.

Altogether, $\{ \alpha\in S\mid f(\eta)(\alpha)>f(\xi)(\alpha)\}$ covers the stationary set $S'$, so that  $f(\eta)\not\sqc{S}f(\xi)$.
\end{proof}

This completes the proof of Theorem~\ref{Sigmacompl}
\end{proof}

Theorem~B now follows as a corollary.
\begin{cor}\label{corollary36} Suppose that $\kappa$ is a regular uncountable cardinal and $\gch$ holds.
Then there is a set-size cofinality-preserving $\gch$-preserving notion of forcing $\mathbb P$,
such that, in $V^{\mathbb P}$, for every analytic quasi-order $Q$ over $\kappa^\kappa$
and every stationary $S\s\kappa$, ${Q}\redulo{\sqc S}$.
\end{cor}
\begin{proof} This follows from Theorems \ref{diamond_from_lcc} and \ref{Sigmacompl}, and one of the following:

$\br$ If $\kappa$ is inaccessible, then we use Fact~\ref{InaccForcing} and Lemma~\ref{slow}.

$\br$ If $\kappa$ is a successor cardinal, then we use Fact~\ref{forcing} and Lemma~\ref{MH2}.
\end{proof}
\begin{remark} By combining the proof of the preceding with a result of L\"ucke \cite[Theorem~1.5]{MR2987148}, we arrive at following conclusion.
Suppose that $\kappa$ is an infinite successor cardinal and $\gch$ holds.
For every binary relation $R$ over $\kappa^\kappa$,
there is a set-size $\gch$-preserving $({<}\kappa)$-closed, $\kappa^+$-cc notion of forcing $\mathbb P_R$
such that, in $V^{\mathbb P_R}$,
the conclusion of Corollary~\ref{corollary36} holds,
and, in addition, $R$ is analytic.
\end{remark}

\begin{remark}A quasi-order $\unlhd$ over a space $X\in\{2^\kappa,\kappa^\kappa\}$ is said to be \emph{$\Sigma^1_1$-complete} iff it is analytic and,
for every analytic quasi-order $Q$ over $X$, there exists a $\kappa$-Borel function $f:X\rightarrow X$ reducing $Q$ to $\unlhd$.
As Lipschitz$\implies$continuous$\implies$$\kappa$-Borel, the conclusion of Corollary~\ref{corollary36} gives that each $\sqc S$ is a $\Sigma^1_1$-complete quasi-order.
Such a consistency was previously only known for $S$'s of one of two specific forms,
and the witnessing maps were not Lipschitz.
\end{remark}

\section{Concluding remarks}
\begin{remark} By \cite[Corollary 4.5]{HKM2}, in $L$, for every successor cardinal $\kappa$
and every theory (not necessarily complete)  $T$ over a countable relational language,
the corresponding equivalence relation $\cong_T$ over $2^\kappa$ is either $\Delta^1_1$ or $\Sigma^1_1$-complete.
This dissatisfying dichotomy suggests that $L$ is a singular universe, unsuitable for studying the correspondence between generalized descriptive set theory and model-theoretic complexities.
However, using Theorem~\ref{Sigmacompl}, it can be verified that the above dichotomy holds as soon as $\kappa$ is a successor of an uncountable cardinal $\lambda=\lambda^{<\lambda}$
in which $\dl^{*}_{S}(\Pi^1_2)$ holds for both $S:=\kappa\cap\cof(\omega)$ and $S:=\kappa\cap\cof(\lambda)$.
This means that the dichotomy is in fact not limited to $L$ and can be forced to hold starting with any ground model.
\end{remark}

\begin{remark}\label{rmk38} Let $=^S$ denote the symmetric version of $\sqc S$.
It is well known that, in the special case $S:=\kappa\cap\cof(\omega)$,
$=^S$ is a $\kappa$-Borel$^*$ equivalence relation \cite[\S6]{MR1242054}.
It thus follows from Theorem~\ref{Sigmacompl} that if $\dl^*_S(\Pi^1_2)$ holds for $S:=\kappa\cap\cof(\omega)$,
then the class of $\Sigma^1_1$ sets coincides with the class of $\kappa$-Borel$^*$ sets.
Now, as the proof of {\cite[Theorem~3.1]{HK18}} establishes that the failure of the preceding is consistent with, e.g., $\kappa=\aleph_2=2^{2^{\aleph_0}}$,
which in turn, by \cite[Lemma~2.1]{MR485361}, implies that $\diamondsuit^*_S$ holds,
we infer that the hypothesis $\dl^*_S(\Pi^1_2)$ of Theorem~\ref{Sigmacompl} cannot be replaced by $\diamondsuit^*_S$.
We thus feel that we have identified the correct combinatorial principle behind a line of results that were previously obtained under the heavy hypothesis of ``$V=L$''.
\end{remark}

\section*{Acknowledgements}
This research was partially supported by the European Research Council (grant agreement ERC-2018-StG 802756). The third author was also partially supported by the Israel Science Foundation (grant agreement 2066/18).

The main results of this paper were presented by the second author at the \emph{4th Arctic Set Theory} workshop, Kilpisj\"arvi, January 2019,
by the third author at the \emph{50 Years of Set Theory in Toronto} conference, Toronto, May 2019,
and by the first author at the \emph{Berkeley conference on inner model theory}, Berkeley, July 2019.
We thank the organizers for the invitations.

\end{document}